\documentclass[12pt,a4paper]{article}
\usepackage[utf8]{inputenc}
\usepackage[T2A]{fontenc}
\usepackage[english]{babel}
\usepackage{amscd,amssymb}
\usepackage{amsfonts,amsmath,array}
\usepackage{amsthm}
\usepackage{longtable,wrapfig}
\usepackage{setspace}
\usepackage{listings}
\usepackage{pgf,tikz}
\usepackage{tikzscale}
\usepackage{mathrsfs}
\usetikzlibrary{arrows}
\PassOptionsToPackage{hyphens}{url}\usepackage{hyperref}
\usepackage[ruled,vlined]{algorithm2e}

\newtheorem{prop}{Proposition}

\newtheorem{cor}{Corollary}

\newtheorem{thm}{Theorem}

\newtheorem{rmk}{Remark}

\title{Constructing 5-chromatic unit distance graphs \\embedded in the Euclidean plane \\and two-dimensional spheres\footnote{Published in Discrete Mathematics,  https://doi.org/10.1016/j.disc.2022.113106}}

\author{V.A. Voronov\footnote{PhD, Researcher, Caucasus Mathematical Center of Adyghe State University, Maikop; Moscow Institute for Physics and Technologies, Moscow, Russian Federation}, A.M. Neopryatnaya\footnote{Undergraduate student, Adyghe State University, Maikop, Russian Federation}, E.A. Dergachev\footnote{Graduate student, Adyghe State University, Maikop, Russian Federation}\\ \vspace{5mm} {\small\textit{v-vor@yandex.ru}}}

\begin{document}

\maketitle

\definecolor{wrwrwr}{rgb}{0.1,0.1,0.1}
\definecolor{qqwuqq}{rgb}{0,0.39216,0}
\definecolor{wrwrwr}{rgb}{0.38039,0.38039,0.38039}
\definecolor{wqwqwq}{rgb}{0.37647,0.37647,0.37647}
\definecolor{qqwuqq}{rgb}{0,0.39,0}
\definecolor{uququq}{rgb}{0.25,0.25,0.25}
\definecolor{wwwwww}{rgb}{0.4,0.4,0.4}

\begin{abstract}
    This paper is devoted to the development of algorithms for finding unit distance graphs with chromatic number greater than 4, embedded in a two-dimensional sphere or plane. Such graphs provide a lower bound for the Hadwiger--Nelson problem  on the chromatic number of the plane and its generalizations to the case of the sphere.
    A series of 5-chromatic unit distance graphs on 64513 vertices embedded into the plane is constructed. Unlike previously known examples, these graphs do not use the Moser spindle as the base element.
   The construction of 5-chromatic graphs embedded in a sphere at two values of the radius is given. Namely, the 5-chromatic unit distance graph on 372 vertices embedded into the circumsphere of an icosahedron with a unit edge length, and the 5-chromatic graph on 972 vertices embedded into the circumsphere of a great icosahedron are constructed.
\end{abstract}


\section{Introduction}

 This paper is devoted to one of the questions related to the Hadwiger--Nelson problem on the chromatic number of the plane \cite{Hadwiger,Soifer}, namely the development of algorithms that allow constructing new examples of 5-chromatic unit distance graphs embedded in a plane or a sphere, based on a given 4-chromatic unit distance subgraph.
 
 The chromatic number of a subset of the Euclidean space $X\subseteq\mathbb{R}^n$  is the minimum number of colors needed to color $X$, so that the points $u$, $v$ at a unit distance apart are colored differently, i.e.
 \[
 \chi(X) = \min \{ \chi \; : \; X = X_1 \sqcup X_2 \sqcup \dots \sqcup X_\chi \quad \forall x,y \in X_i \quad  \|x-y| \neq 1 \}.
 \]
 
 P.~Erdős posed a number of problems on chromatic numbers of higher dimensional spaces and on the asymptotics of chromatic numbers of $\mathbb{R}^n$ and $S^n(r)$, $r>1/2$ as $n \rightarrow \infty$. These problems offer opportunities for more complicated and meaningful mathematical tools, and many more papers are devoted to them compared to the two-dimensional case.
 
 Let us mention some results for higher dimensions. Currently, the best lower estimate for the three-dimensional space $$\chi(\mathbb{R}^3)\geq 6$$ belongs to O.~Nechushtan \cite{Nechushtan}. New lower estimates for $\chi(\mathbb{R}^n)$, $9\leq n\leq 12   $ were obtained by Cherkashin et al. in \cite{Rai3}.

 D.~Larman and C.~Rogers \cite{Larman} proved that
 \[
\chi(\mathbb{R})^n \leq (3 +o(1))^n.
\]

 A.~M.~Raigorodskii showed in \cite{Rai1} that
 \[
\chi(\mathbb{R})^n \geq (1.239... +o(1))^n.
\]

In \cite{Lovasz} L.~Lovász showed that for $n \geq 3$ there exists an infinite monotone decreasing sequence of radii \[\{r^{(n)}_i\}, \quad \lim_{i \to \infty} r^{(n)}_i = \frac{1}{2}\] such that
\[
\chi(S^{n-1}(r^{(n)}_i)) = n+1.
\]
Lovász's paper gives an explicit construction of \emph{strongly self-dual polytopes} whose graphs of diameters provide the above estimate.    Since  $S^{n-1}(r_0) \subset S^{n}( r)$ when $r \geq r_0$, the consequence is the estimate
\[
\chi(S^{n-1}(r)) \geq n, \quad r > \frac{1}{2}.
\]

In \cite{Rai2} Raigorodskii proved  that the chromatic number of a $n$-dimensional sphere of radius $r$ grows exponentially for any $r>\frac{1}{2}$, correcting the statement mentioned in \cite{Lovasz}.
 
 Recently, O.~A.~Kostina has improved estimates for the asymptotics of $\chi(S^{n-1}(r))$ at $r>1/2$ \cite{Kostina}. 
 
 R.~Prosanov obtained new upper estimates for the chromatic numbers of spheres $S^{n-1}(r)$ and balls $B^n(r)$ \cite{Prosanov}.

  According to the de Bruijn--Erdős theorem \cite{Erdos}, for any $X\subset\mathbb{R}^n$ there exists a finite set of points $V \subseteq X$ for which $\chi(X)=\chi(V)$. Therefore the direct way to obtain lower estimates of $\chi(X)$ is to construct a finite graph with the maximal possible chromatic number whose vertices are points of $X$ and vertices that are at distance 1 apart are connected. Such graphs are called unit distance graphs. 

\begin{figure}
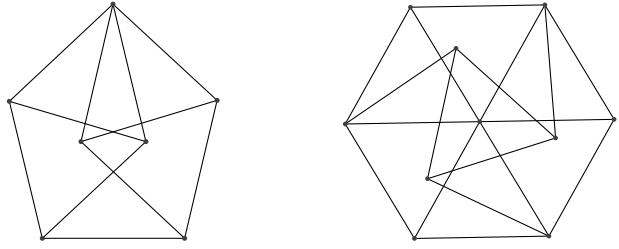

    \centering
\begin{tabular}{p{4cm} p{4cm} } 

\includegraphics[width=2.8cm]{moser} &\includegraphics[width=3.6cm]{golomb}
 \\
\end{tabular}    
    \caption{The Moser spindle and the Golomb graph}
    \label{moser_golomb}
\end{figure}
    For the plane, the  lower estimate of $\chi(\mathbb{R}^2)\geq 4$ associated with the Moser spindle, a 4-chromatic graph on 7 vertices (Fig. \ref{moser_golomb}), has long been the best. 
    The first example of a 5-chromatic unit distance graph constructed by A.~de~Grey in 2018 contained 1581 vertices and 7877 edges \cite{deGrey}. A little later, G. Exoo and D. Ismailescu published a paper with a significantly different construction of the 5-chromatic graph~\cite{Exoo2}. M.~Heule \cite{Heule1} modified de Grey's construction and used  massive parallel computations to minimize the number of vertices in the graph. He managed to reduce the number of vertices to 553, then to 517. Then J. Parts, using a different vertex selection algorithm, found a 5-chromatic distance graph with 509 vertices, which is by far the minimal number of vertices in the case of the plane \cite{Parts1}. In addition, he was able to propose a proof of the estimate $\chi(\mathbb{R}^2)\geq 5$ in which the volume of enumeration is available for direct human verification \cite{Parts2}. One can say that the graph has been constructed for which, when checking the absence of 4-coloring, the depth-first search  passes much less nodes than for the 5-chromatic graphs considered earlier.

The de Grey construction, its modification proposed by M.~Heule, and the graph by J.~Parts are based on a graph containing a large number of copies of the Moser spindle as subgraphs. The construction of G.~Exoo and D.~Ismailescu was more complicated, but the resulting graph also contained the Moser spindle, the simplest 4-chromatic unit distance graph. Hence a natural question arises. 

Can a 5-chromatic distance graph be constructed using another 4-chro\-ma\-tic distance graph with a small number of vertices instead of the Moser spindle? The answer is positive, and the present paper gives such a construction. 

In addition, constructions based on other small 4-chromatic distance graphs allowed us to find 5-chromatic graphs embedded in the sphere $S^2(r)$ at two values of the sphere radius. Apparently, such constructions have not been found before in the published works.

The previously known lower and upper estimates for the chromatic number of the two-dimensional sphere are given in Table \ref{sph_estimates}. Most of these results belong to G.~Simmons~\cite{Simmons}. Recently, some results have been obtained on sphere colorings in which connected regions painted in the same color are bounded by arcs of circles. P. Ágoston proved that if $r \geq 18$ then 7 colors are not enough to obtain a map-type coloring based on the coloring of a square of a dual triangulation~\cite{Agoston}. T. Sirgedas showed that 7-coloring exists if a radius is large enough \cite{Sirgedas}.

\begin{table}[ht]
    \centering
    \begin{tabular}{|c|c|c|}
    \hline
        $r$ & Estimate for $\chi(r)=\chi(S^2(r))$ & Source  \\
        \hline
        $r<1/2$ & $\chi(r)=1$ &   \\
         \hline
        $r = 1/2$ & $\chi(r)=2$ & \\
         \hline
        $r >1/2$ & $\chi(r)\geq 3$ & \cite{Simmons} \\
         \hline
        $r = \sqrt{2}/2$ & $\chi(r)=4$ & \cite{Simmons,Godsil} \\
         \hline
        $r\geq \sqrt{3}/3$ & $\chi(r)\geq 4$ & \cite{Simmons} \\
         \hline
        $ r>1/2$ & $\chi(r)\leq 15$ & \cite{Coulson} \\
         \hline
        $ r \geq 12.44$ & $\chi(r) \leq 7$ & \cite{Sirgedas} \\
         \hline
    \end{tabular}
    \caption{Lower and upper estimates for $\chi(S^2(r))$.}
    \label{sph_estimates}
\end{table}

\section{Small minimally rigid unit distance graphs}

Let $V \subset \mathbb{R}^n$ be an arbitrary set of points. Denote as 
\[
G=(V,E)=\operatorname{udg}(V)
\]
a (faithful) unit distance graph \cite{Alon} with vertices from $V$,
for which
\begin{equation*}
    E = \{ (u,v): u,v \in V, |u-v| = 1 \}.
\end{equation*}

Following Graver et al., we will call a unit distance graph \emph{rigid} if there is no continuous motion of its vertices for which the lengths of all the edges are preserved and at least one of the pairwise distances between vertices changes \cite{combrig}. Note that the 4-chromatic graph $L_{17}$ on 17 vertices  with girth 4, constructed in \cite{Exoo2} is rigid.  In addition, the Moser spindle and $L_{17}$ are \emph{minimally rigid} \cite{Laman,Henneberg}, i.e., they lose the rigidity property after removing any edge. Moreover, the chromatic number of these graphs decreases after removing any vertex (i.e. they are 4-vertex-critical). The Golomb graph (Fig. \ref{moser_golomb}) also contains a 4-chromatic minimally rigid subgraph. Hence an empirical idea arises, which guided us in further work.

\medskip
\textbf{Heuristic}. Below we will consider minimally rigid 4-vertex-critical graphs with a small number of vertices  as a basic element for constructing a new example of a 5-chromatic graph.

\medskip
The number of abstract graphs with these properties is given in Table \ref{laman}. Here $l_{4,k}$ is the number of minimally rigid 4-chromatic graphs on $k$ vertices~\cite{oeis2} and $l'_{4,k}$ is the number of minimally rigid 4-vertex-critical graphs.

\begin{table}[ht]
\begin{center}
\begin{tabular}{ |c|c|c|c|c|c|c|c|c|c|c| } 
\hline
$k$ \, &        7   & 8   & 9 & 10 & 11 \\
\hline 
$l_{4,k}$  & 1   & 8    & 102 & 1601 & 29811 \\

\hline 
$l'_{4,k}$  & 1   & 1  & 6 & 60 & 241  \\

\hline 
\end{tabular}
\caption{Number of unlabeled 4-chromatic minimally rigid graphs}
\end{center}
\label{laman}
\end{table}

A complete enumeration of the distance embeddings of these graphs up to $n=10$ is not a difficult task, but in this paper we will limit ourselves to considering only a few of them (Fig.~\ref{smallgraphs}).

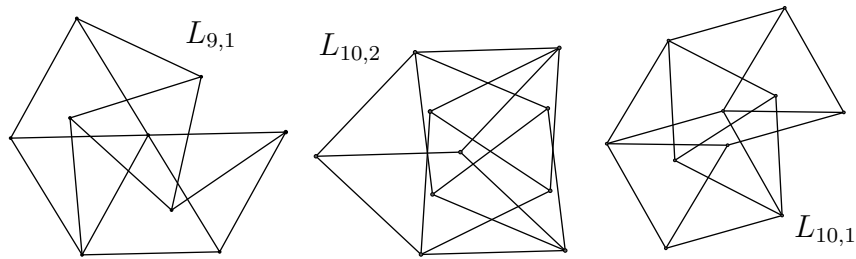
\begin{figure}[ht]
\begin{center}
\begin{tabular}{ p{3.5cm} p{3.5cm} p{3.5cm} } 
\begin{tikzpicture}[line cap=round,line join=round,>=triangle 45,x=1cm,y=1cm, scale=0.5]
\draw (1.0,4.0) node[anchor=north west]  {$L_{9,1}$};
\draw [line width=0.5pt] (-3.3,0.6)-- (-1.5592820323027552,3.775011961699667);
\draw [line width=0.5pt] (-1.5592820323027552,3.775011961699667)-- (0.32,0.68);
\draw [line width=0.5pt] (0.32,0.68)-- (-3.3,0.6);
\draw [line width=0.5pt] (-3.3,0.6)-- (-1.4207179676972448,-2.4950119616996673);
\draw [line width=0.5pt] (-1.4207179676972448,-2.4950119616996673)-- (0.32,0.68);
\draw [line width=0.5pt] (-1.4207179676972448,-2.4950119616996673)-- (2.1992820323027553,-2.4150119616996673);
\draw [line width=0.5pt] (2.1992820323027553,-2.4150119616996673)-- (0.32,0.68);
\draw [line width=0.5pt] (0.32,0.68)-- (3.94,0.76);
\draw [line width=0.5pt] (3.94,0.76)-- (2.1992820323027553,-2.4150119616996673);
\draw [line width=0.5pt] (-1.7378163968074531,1.1353767695399808)-- (1.7143628565417033,2.2277404465018504);
\draw [line width=0.5pt] (1.7143628565417033,2.2277404465018504)-- (0.9342879242874824,-1.3081163237970497);
\draw [line width=0.5pt] (0.9342879242874824,-1.3081163237970497)-- (-1.7378163968074531,1.1353767695399808);
\draw [line width=0.5pt] (-1.5592820323027552,3.775011961699667)-- (1.7143628565417033,2.2277404465018504);
\draw [line width=0.5pt] (0.9342879242874824,-1.3081163237970497)-- (3.94,0.76);
\draw [line width=0.5pt] (-1.7378163968074531,1.1353767695399808)-- (-1.4207179676972448,-2.4950119616996673);
\begin{scriptsize}
\draw [fill=black] (-3.3,0.6) circle (1pt);
\draw [fill=black] (0.32,0.68) circle (1pt);
\draw [fill=black] (-1.4207179676972448,-2.4950119616996673) circle (1pt);
\draw [fill=black] (2.1992820323027553,-2.4150119616996673) circle (1pt);
\draw [fill=black] (-1.5592820323027552,3.775011961699667) circle (1pt) ;
\draw [fill=black] (3.94,0.76) circle (1.3pt);
\draw [fill=black] (1.7143628565417033,2.2277404465018504) circle (1pt);
\draw [fill=black] (-1.7378163968074531,1.1353767695399808) circle (1pt);
\draw [fill=black] (0.9342879242874824,-1.3081163237970497) circle (1pt) ;
\end{scriptsize}
\end{tikzpicture} & \begin{tikzpicture}[line cap=round,line join=round,>=triangle 45,x=1cm,y=1cm, scale=0.65]
\draw(-4.5,2.0)  node[anchor=north west] {$L_{10,2}$};
\draw [line width=0.5pt] (0.4166303387716028,-1.4415217867749939)-- (-2.2,-2.74);
\draw [line width=0.5pt] (-2.2,-2.74)-- (-2.016199949527681,0.17530745230179287);
\draw [line width=0.5pt] (-2.016199949527681,0.17530745230179287)-- (0.4166303387716028,-1.4415217867749939);
\draw [line width=0.5pt] (0.4166303387716028,-1.4415217867749939)-- (0.6004303892439218,1.473785665526799);
\draw [line width=0.5pt] (0.6004303892439218,1.473785665526799)-- (-2.016199949527681,0.17530745230179287);
\draw [line width=0.5pt] (0.36799789118532045,0.23980939294120096)-- (-1.9673096548268512,-1.5149380719479018);
\draw [line width=0.5pt] (-1.9673096548268512,-1.5149380719479018)-- (-2.3193117636415304,1.3848713209932992);
\draw [line width=0.5pt] (-2.3193117636415304,1.3848713209932992)-- (0.36799789118532045,0.23980939294120096);
\draw [line width=0.5pt] (0.36799789118532045,0.23980939294120096)-- (0.72,-2.66);
\draw [line width=0.5pt] (0.72,-2.66)-- (-1.9673096548268512,-1.5149380719479018);
\draw [line width=0.5pt] (-2.3193117636415304,1.3848713209932992)-- (-4.3265439087208994,-0.7373490058719567);
\draw [line width=0.5pt] (-4.3265439087208994,-0.7373490058719567)-- (-2.2,-2.74);
\draw [line width=0.5pt] (0.6004303892439218,1.473785665526799)-- (-1.4022156496253484,-0.6527629097475159);
\draw [line width=0.5pt] (-1.4022156496253484,-0.6527629097475159)-- (0.72,-2.66);
\draw [line width=0.5pt] (-1.4022156496253484,-0.6527629097475159)-- (-4.3265439087208994,-0.7373490058719567);
\draw [line width=0.5pt] (-2.2,-2.74)-- (0.72,-2.66);
\draw [line width=0.5pt] (-2.3193117636415304,1.3848713209932992)-- (0.6004303892439218,1.473785665526799);
\begin{scriptsize}
\draw [fill=wrwrwr] (-2.2,-2.74) circle (1pt);
\draw [fill=wrwrwr] (0.72,-2.66) circle (1pt);
\draw [fill=wrwrwr] (0.4166303387716028,-1.4415217867749939) circle (1pt);
\draw [fill=wrwrwr] (-2.016199949527681,0.17530745230179287) circle (1pt);
\draw [fill=wrwrwr] (0.6004303892439218,1.473785665526799) circle (1pt);
\draw [fill=wrwrwr] (-2.3193117636415304,1.3848713209932992) circle (1pt);
\draw [fill=wrwrwr] (-1.9673096548268512,-1.5149380719479018) circle (1pt);
\draw [fill=wrwrwr] (0.36799789118532045,0.23980939294120096) circle (1pt);
\draw [fill=wrwrwr] (-2.016199949527681,0.17530745230179287) circle (1pt);
\draw [fill=wrwrwr] (-1.4022156496253484,-0.6527629097475159) circle (1pt);
\draw [fill=wrwrwr] (-4.3265439087208994,-0.7373490058719567) circle (1pt);
\end{scriptsize}
\end{tikzpicture} & \begin{tikzpicture}[line cap=round,line join=round,>=triangle 45,x=1cm,y=1cm, scale=0.5]
\draw (0.5,-5.3)  node[anchor=north west] {$L_{10,1}$};
\draw [line width=0.5pt] (-4.14,-3.68)-- (-0.96,-3.72);
\draw [line width=0.5pt] (-4.14,-3.68)-- (-1.0808035764605388,-2.810910106949017);
\draw [line width=0.5pt] (-1.0808035764605388,-2.810910106949017)-- (2.0991964235394605,-2.850910106949017);
\draw [line width=0.5pt] (2.0991964235394605,-2.850910106949017)-- (-0.96,-3.72);
\draw [line width=0.5pt] (-1.0808035764605388,-2.810910106949017)-- (0.5438374396908383,-0.07694932291450265);
\draw [line width=0.5pt] (0.5438374396908383,-0.07694932291450265)-- (2.0991964235394605,-2.850910106949017);
\draw [line width=0.5pt] (0.5438374396908383,-0.07694932291450265)-- (-2.515358983848622,-0.9460392159654859);
\draw [line width=0.5pt] (-2.515358983848622,-0.9460392159654859)-- (-4.14,-3.68);
\draw [line width=0.5pt] (-4.14,-3.68)-- (-2.584641016151377,-6.453960784034515);
\draw [line width=0.5pt] (-0.96,-3.72)-- (-2.584641016151377,-6.453960784034515);
\draw [line width=0.5pt] (0.4745554073880833,-5.584870890983531)-- (-1.0808035764605388,-2.810910106949017);
\draw [line width=0.5pt] (0.4745554073880833,-5.584870890983531)-- (-2.584641016151377,-6.453960784034515);
\draw [line width=0.5pt] (0.3084436564034514,-2.408960488250983)-- (-2.3492472328639904,-4.121949618698035);
\draw [line width=0.5pt] (-2.3492472328639904,-4.121949618698035)-- (-2.515358983848622,-0.9460392159654859);
\draw [line width=0.5pt] (-2.515358983848622,-0.9460392159654859)-- (0.3084436564034514,-2.408960488250983);
\draw [line width=0.5pt] (0.3084436564034514,-2.408960488250983)-- (0.4745554073880833,-5.584870890983531);
\draw [line width=0.5pt] (0.4745554073880833,-5.584870890983531)-- (-2.3492472328639904,-4.121949618698035);
\begin{scriptsize}
\draw [fill=wrwrwr] (-4.14,-3.68) circle (1pt);
\draw [fill=wrwrwr] (-0.96,-3.72) circle (1pt);
\draw [fill=wrwrwr] (-1.0808035764605388,-2.810910106949017) circle (1pt);
\draw [fill=wrwrwr] (2.0991964235394605,-2.850910106949017) circle (1pt);
\draw [fill=wrwrwr] (0.5438374396908383,-0.07694932291450265) circle (1pt);
\draw [fill=wrwrwr] (-2.515358983848622,-0.9460392159654859) circle (1pt);
\draw [fill=wrwrwr] (-2.584641016151377,-6.453960784034515) circle (1pt);
\draw [fill=wrwrwr] (0.4745554073880833,-5.584870890983531) circle (1pt);
\draw [fill=wrwrwr] (0.3084436564034514,-2.408960488250983) circle (1pt);
\draw [fill=wrwrwr] (-2.3492472328639904,-4.121949618698035) circle (1pt);
\end{scriptsize}
\end{tikzpicture}
 \\
\end{tabular}
\end{center}
\caption{Minimally rigid 4-chromatic unit distance graphs}
\label{smallgraphs}
\end{figure}

The graph $L_{9,1}$, which is a subgraph of the Golomb graph, will be used in the construction of a 5-chromatic graph on the circumsphere of the icosahedron. The graph $L_{10,1}$ is used in a similar construction for the circumsphere of a great icosahedron. The graph $L_{10,2}$ is an element of the construction of a new example of a 5-chromatic graph embedded in the plane. 

\section{Gluing distance graphs by edges}

Let $u_1,v_1,u_2,v_2 \in S^2(r)$ be such points that $|u_1-v_1|=|u_2-v_2|=1$ and $r > 1/2$. Obviously, there are exactly four isometries of $S^2(r)$ translating the pair $\{u_2, v_2\}$ to $\{u_1, v_1\}$. Denote the set of these isometries by $$P(u_1,v_1; u_2,v_2)\subset O(3).$$

Consider distance graphs $G_1=(V_1,E_1)$ and $G_2 = (V_2, E_2)$; $V_i \subset S^2(r)$, $i=1,2$. 

\begin{figure}[ht]
\begin{center}
\includegraphics[width=12cm]{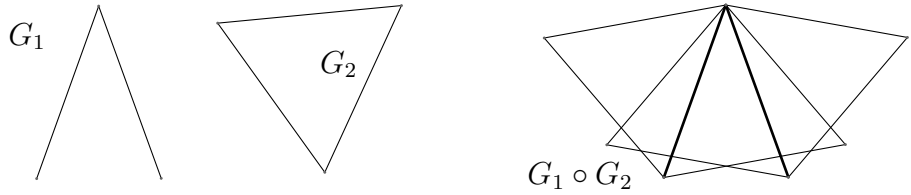}
\end{center}
\caption{Distance product}
\label{distprod}
\end{figure}

Let
\begin{equation*}
\mathcal{F}(G_1, G_2) = \bigcup\limits_{(u_1,v_1) \in E_1, \; (u_2,v_2) \in E_2} P(u_1, v_1; u_2, v_2).
\end{equation*}

Like before, denote the unit distance graph with the set of vertices $V$ by $\operatorname{udg}(V)$, $V \subset S^2(r)$.

Let us call the distance graph
\begin{equation*}
    G_1 \circ G_2 = \operatorname{udg}\left( \bigcup\limits_{f \in \mathcal{F}(G_1,G_2)} f(V_2) \right)
\end{equation*}

by ``distance product'' of distance graphs $G_1$ and $G_2$ (Fig. \ref{distprod}). Note that in the general case $G_1 \circ G_2$ is not isomorphic to $G_2 \circ G_1$ even as an abstract graph. 

Obviously, if the aim is to construct a distance graph with chromatic number $k$, then one can discard all vertices $v\in V(G)$, $\operatorname{deg} v < k-1$, and repeat this operation until either the graph becomes empty, or the degrees of all of the remaining vertices are not less than $k-1$. In some cases to simplify the construction of $k$-chromatic graph one can use a tighter restriction for the degree of vertices.

\section{The case of the sphere}

Consider a sphere $S^2(r_1)$ in which an icosahedron with a unit edge length is inscribed. Its radius is 
$$r_1 = \cos \frac{\pi}{10} = \frac{\sqrt{5+\sqrt{5}}}{2\sqrt{2}} =  0.95105...$$

\begin{figure}[ht]
    \begin{center}
    \begin{tabular}{ p{4.5cm} p{4.5cm} } 
        \includegraphics[width=3.5cm]{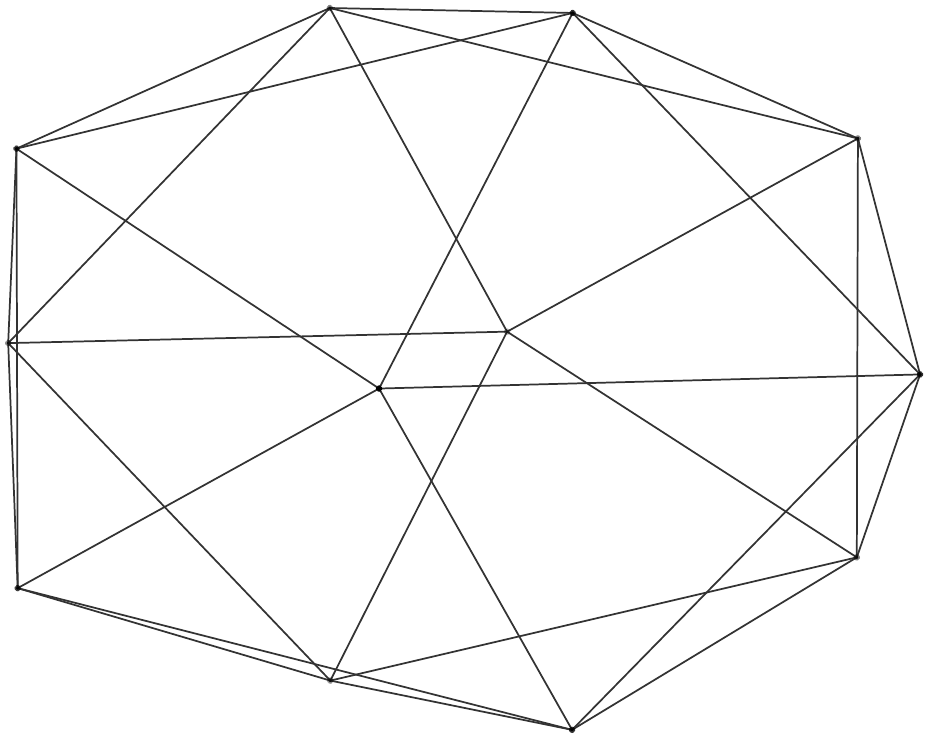} & \includegraphics[width=4cm]{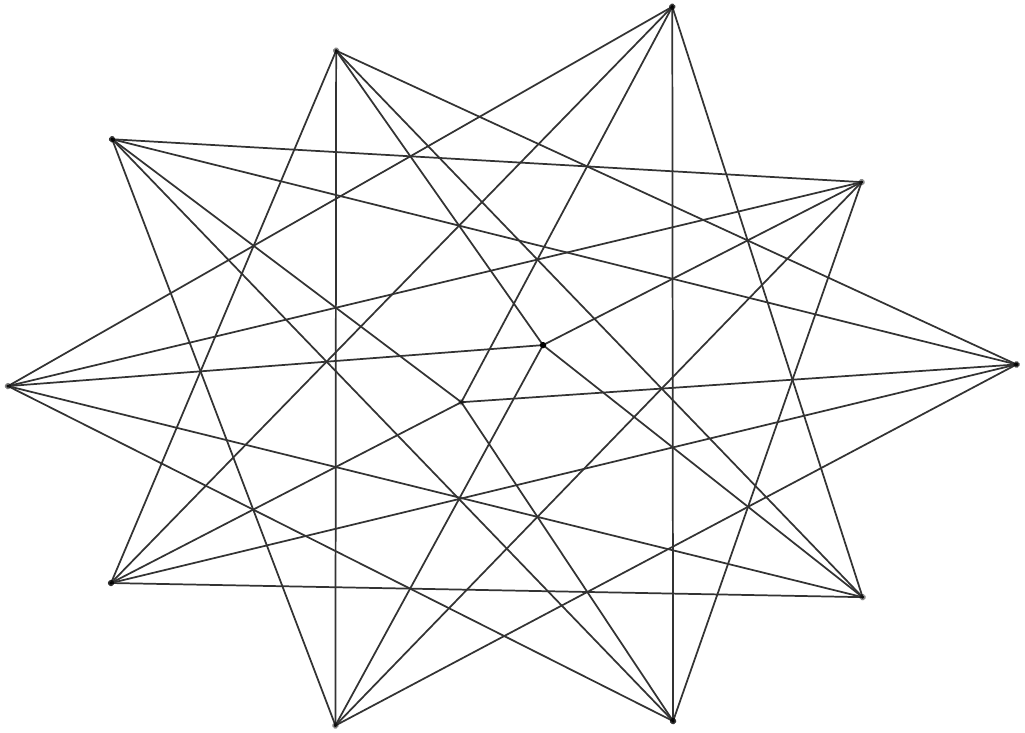}
    \end{tabular}
    \end{center}
    \caption{Distance graphs $H_{12,1}$ and $H_{12,2}$}
    \label{icos}
\end{figure}

Let us denote  the graph of unit distances generated by the vertices of the icosahedron by $H_{12,1}$ (Fig. \ref{icos}, left). Denote by $\alpha(G)$ the independence number (i.e. the cardinality of a largest independent set) of the graph $G$. It is easy to see that

\begin{prop}
$$\alpha(H_{12,1}) = 3;$$
$$\chi(H_{12,1}) = 4.$$
\end{prop}

Moreover, the proper 4-coloring of $H_{12,1}$  is unique with respect to the action of the symmetry group of the icosahedron and color permutations \cite{Ballard}. In fact, any triplet of pairwise non-adjacent vertices forms an equilateral triangle with side $\tau = \frac{\sqrt{5}+1}{2}$, and four such triangles without common vertices must lie in planes which form a regular tetrahedron. 

This imposes restrictions on the possible coloring of $S^2(r_1)$ in 4 colors (in particular, the opposite points of $S^2(r_1)$ have different colors).

\begin{prop}
The graph $L_{9,1}$ can be embedded in $S^2(r_1)$ as a unit distance graph.
\end{prop}

 Let the vertices $v_1, \dots , v_6$ forming four equilateral triangles with a common vertex coincide with the ``pentagonal cap'' of the icosahedron, and  the icosahedron vertex coordinates are
 \[
 \left(0, \pm\frac{\tau}{2}, \pm\frac{1}{2} \right), \quad \left(\pm\frac{\tau}{2}, \pm\frac{1}{2}, 0 \right), \quad \left(  \pm\frac{1}{2}, 0, \pm\frac{\tau}{2} \right).
 \]
 
 \begin{table}[ht]
    \centering
    \begin{tabular}{|c|c|c|}
    \hline
    & Value & Minimal polynomial for $2 x_i$, $2 y_i$, $2 z_i$,\\
    \hline
    $x_{1}$ & 0.71584   & $-1 - x + 2 x^2 - 2 x^3 + x^4$ \\
    $y_{1}$ & 0.34924   &  $-1 + 2 x - 2 x^2 + x^3 + x^4 $\\
    $z_{1}$ & 0.51972   & $-19 + 20 x^2 + 2 x^4 - 5 x^6 + x^8$\\
    \hline
    $x_{2}$ & -0.07961   & $1 + 8 x + 9 x^2 - 7 x^3 + 28 x^4 + 6 x^5 - 4 x^6 - x^7 + x^8$ \\
    $y_{2}$ & -0.08345   & $-1 - 5 x + 5 x^2 - 4 x^3 + 9 x^4 - 5 x^5 + x^6 - 2 x^7 + x^8$ \\
    $z_{2}$ & 0.94404   & $-1 + 7 x - 13 x^2 - 24 x^3 + 4 x^4 + 10 x^5 - 3 x^7 + x^8$ \\
    \hline
    $x_{3}$ & -0.17646   & $1 + 8 x + 9 x^2 - 7 x^3 + 28 x^4 + 6 x^5 - 4 x^6 - x^7 + x^8$ \\
    $y_{3}$ & 0.79929   & $- 1 - 5 x + 5 x^2 - 4 x^3 + 9 x^4 - 5 x^5 + x^6 - 2 x^7 + x^8$ \\
    $z_{3}$ & 0.48426   & $-1 - 7 x - 13 x^2 + 24 x^3 + 4 x^4 - 10 x^5 + 3 x^7 + x^8$ \\
    \hline
    \end{tabular}
    \caption{Minimal polynomials for vertex coordinates  of $G_{372}$}
    \label{coord1}
\end{table}
 
  Then the coordinates of the other three vertices of $L_{9,1}$ are given in the Table \ref{coord1}. The projection of this embedding is shown in Fig. \ref{embed} (left). Note that when $L_{9,1}$ is embedded in the sphere of radius $r_1$, there appears a unit distance, which is not an edge of $L_{9,1}$ (shown by the dashed line).

\begin{figure}[ht]
    \begin{center}
    \begin{tabular}{ p{4.5cm} p{4.5cm} } 
        \includegraphics[width=3.5cm]{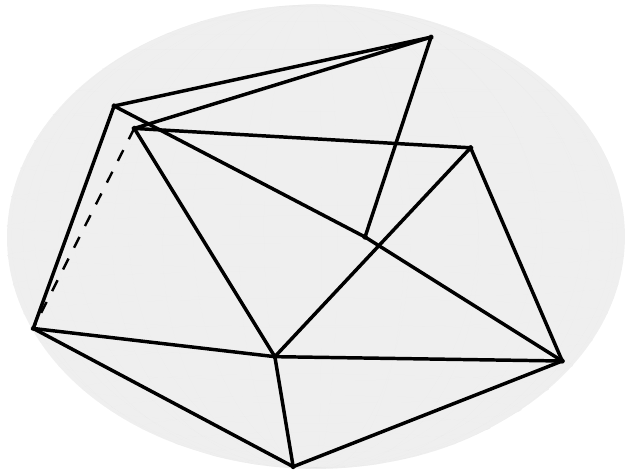} & \includegraphics[width=3.5cm]{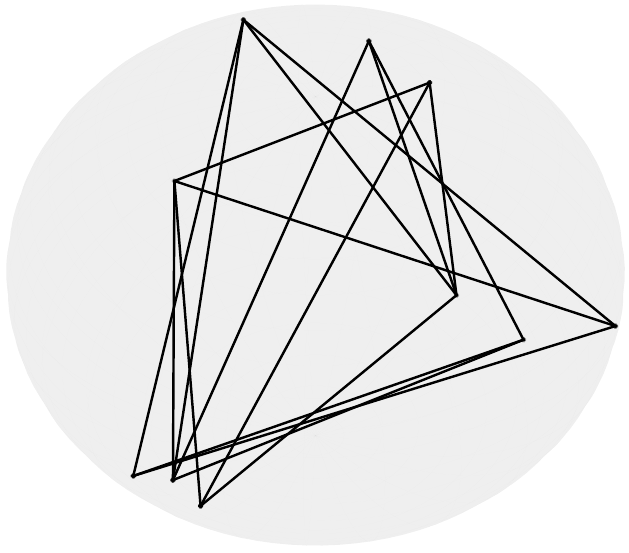}
    \end{tabular}
    \end{center}
    \caption{Embeddings of $L_{9,1}$ and $L_{10,1}$}
    \label{embed}
\end{figure}

\begin{figure}
    \centering
    \includegraphics[width=10cm]{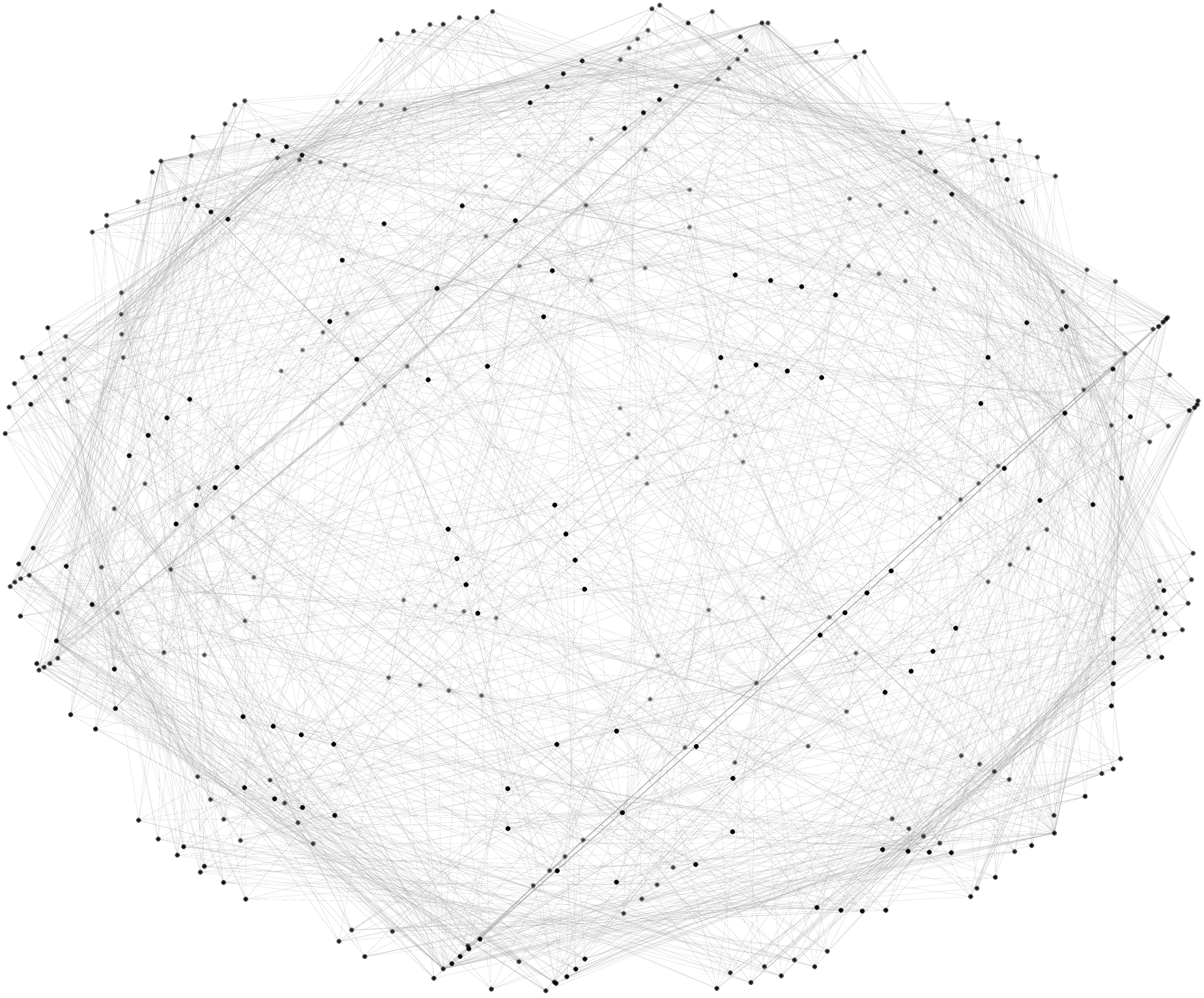}
    \caption{5-chromatic 372-vertex unit distance graph embedded into the circumsphere of the icosahedron}
    \label{g372}
\end{figure}

\begin{figure}
    \centering
    \includegraphics[width=10cm]{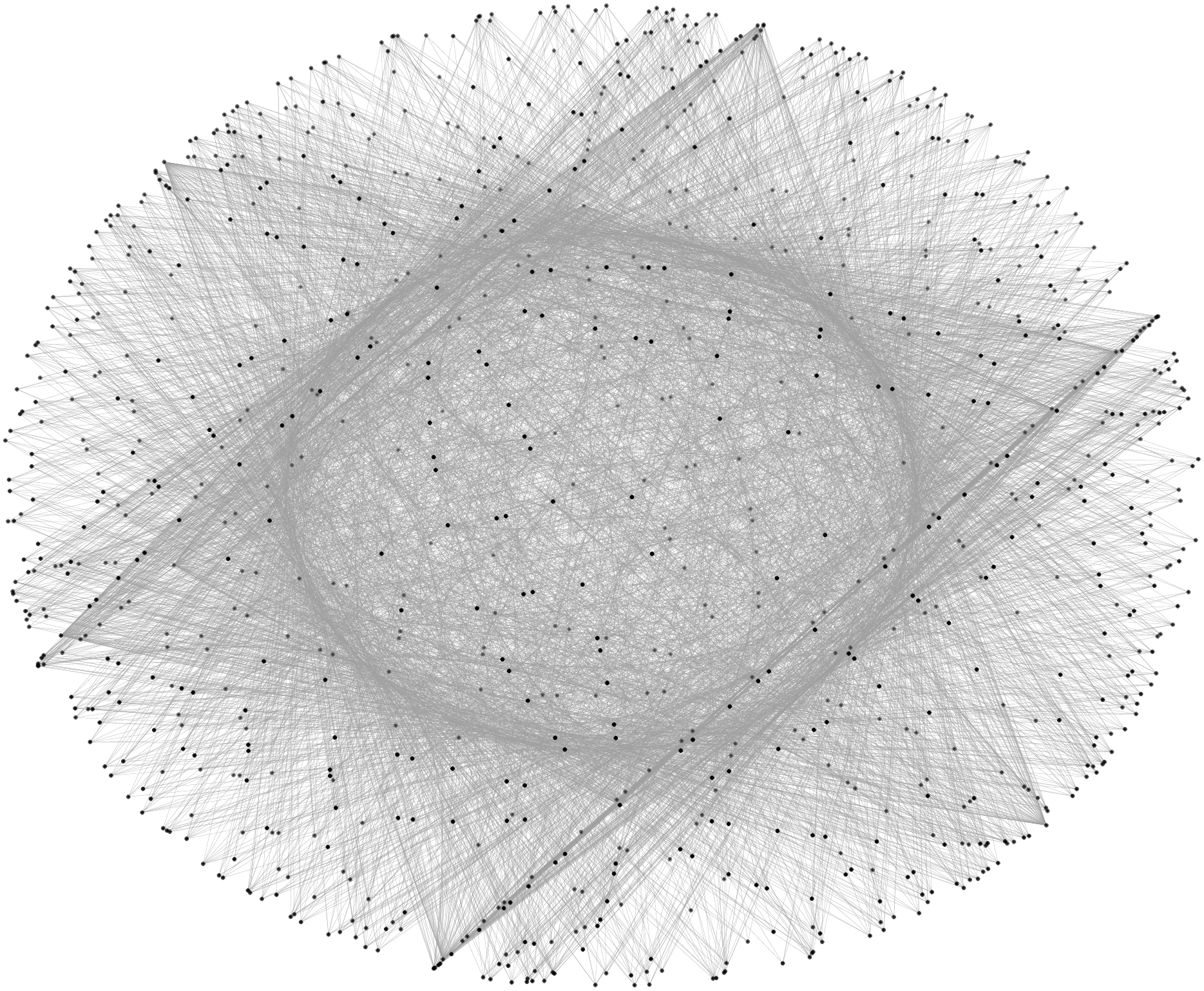}
    \caption{5-chromatic 972-vertex unit distance graph embedded into the circumsphere of the great icosahedron}
    \label{g972}
\end{figure}

\begin{prop}
Let $H_{12,1}$ be a distance embedding of the icosahedron graph in $S^2(r)$ and $H_{9,1}$ is the embedding of $L_{9,1}$. Then for the ``distance product'' we have
\begin{equation*}
    \chi\left( H_{12,1} \circ H_{9,1} \right) = 5.
\end{equation*}
\end{prop}

The graph $H_{12,1} \circ H_{9,1}$ contains 732 vertices and 3390 edges. Preserving the symmetry group and the chromatic number, we can discard vertices of degree 7 and less. The resulting subgraph has 372 vertices and 1710 edges; denote it by $G_{372}$ (Fig. \ref{g372}). There are  4 equivalence classes of vertices under the action of the symmetry group of the icosahedron. The vertices of the original icosahedron have degree 35, in addition, the graph contains 120 vertices of degree 9, and 240 vertices of degree 8. As representatives of last three equivalence classes, one can take the vertices whose coordinates are given in Table \ref{coord1}.  The set of vertices can also be constructed by combining 37 copies of the icosahedron vertices $V(H_{12,1})$ obtained by rotations at angles $\pm\alpha_1, \pm\alpha_2, \pm\alpha_3$ around axes passing through pairs of opposite vertices, where 
\[
\alpha_1 = 0.20895517... , \alpha_2 = 0.53459809..., \alpha_3 = 0.59935172...
\]

It takes less than 1 second to check the absence of a 4-coloring of $G_{372}$  using the classical SAT solver Minisat or the IGraph/M library in Wolfram Mathematica \cite{igraphm}. If we paint the vertices sequentially starting from the 4-colored icosahedron, and at each step choose a vertex for which the number of choices is minimal, then the search tree contains  $\approx 2\cdot10^5$ nodes.  It is possible that here the search can be reduced to the quantity available for manual verification, as it was done in \cite{Parts1} for the case of the plane.

It was found that the chromatic number of a graph $G_{372}$ preserves after deleting any vertex, including the vertex of the initial icosahedron, but decreases to 4 after deleting any class of vertices. We do not give here the subgraphs of $G_{372}$ that have the property of minimality in the sense of conservation of the chromatic number.

A similar, but somewhat more complicated construction was performed for the circumscribed sphere of the great icosahedron with unit edge length. Fig. \ref{embed} (right) shows a distance embedding of the graph $L_{10,1}$ in a sphere of radius
$$r_2 = \cos \frac{3\pi}{10} = \frac{\sqrt{5-\sqrt{5}}}{2\sqrt{2}} = 0.58778...$$ 

\begin{prop}
Let us denote the distance graph formed by vertices of a great icosahedron by $H_{12,2}$ (Fig. \ref{icos}, right). Let $H_{10,1}$ be the embedding of $L_{10,1}$ in a sphere of radius $r_2$ (Fig. \ref{embed}, right). Then
\begin{equation*}
    \chi\left( H_{12,2} \circ H_{10,1} \right) = 5.
\end{equation*}
\end{prop}

The graph $H_{12,2} \circ H_{10,1}$ contains 3132 vertices and 10230 edges. After iterative removal of vertices of degree at most 4 we have a graph $G_{972}$ with 972 vertices and 4110 edges (Fig. \ref{g972}). The minimal polynomials for doubled vertex coordinates for each of the representatives of 8 equivalence classes (except the vertices of the great icosahedron) are given in Table \ref{coord2}. As in the previous case, the test for absence of 4-coloring is performed fast enough, less than 5 seconds on a single core of 4 GHz Intel CPU.

The repository \cite{Github} contains Wolfram Mathematica notebooks in which distance graphs $G_{372}$, $G_{972}$ are constructed using analytically given vertex coordinates, and their chromatic number is found using the IGraph/M library~\cite{igraphm}.  

Finally,  we have the following estimates of the chromatic numbers of two-dimensional spheres.

\begin{thm}
\begin{equation*}
    \chi\left(S^2(r_1)\right) \geq 5; \quad r_1 = \cos \frac{\pi}{10} = 0.95105...
\end{equation*}
\begin{equation*}
    \chi\left(S^2(r_2)\right) \geq 5; \quad r_2 = \cos \frac{3\pi}{10} = 0.58778...
\end{equation*}
\end{thm}

In addition, since $\chi\left(S^2\left(\frac{\sqrt{2}}{2}\right)\right)=4$, we have the following 
\begin{cor}
The function $\chi(r) = \chi(S^2(r))$ is not monotonic.
\end{cor}

\section{Construction for the plane}

For the sake of brevity we denote the Moser spindle by $L_7$. The construction of the 5-chromatic unit distance graph by M. Heule was based on the following idea: one should consider sets of vertices for which the corresponding unit distance graph contains a large number of $L_7$ as subgraphs and then construct a union of this set with its copy rotated by some angle around the common symmetry center (the origin). In this section it will be shown that instead of  $L_7$, another 4-chromatic unit distance graph can be chosen. We are going to generalize the construction by M. Heule \cite{Heule1}, consider other small 4-chromatic graphs as a starting point and perform enumeration of rotations moving one of the subsets of vertices to another. 

For convenience we will further consider graph embeddings in $\mathbb{C}$ instead of embeddings in $\mathbb{R}^2$.

Let $L$ be a 4-chromatic rigid unit distance graph. Consider such an embedding of $L$ in $\mathbb{C}$  when the ends of some edge $(u,v) \in E(L)$ are points $u=0$ and $v=1$. 

Then consider the finite set  
\[
D(L) = \{u-v \; | \; (u,v) \in E(L) \}
\]
and a finitely generated group consisting of all possible products of elements of $D(L)$ and their inverses
\[
g(L) = \langle D(L) \rangle.
\]
Let us choose some minimal set of generators 
\[
\Phi(L)=\{\phi_0, \dots , \phi_p\}, \quad \langle \Phi(L) \rangle = g(L).
\]

Here $\phi_0$ is a finite-order element, $\phi_0^k=1$. In the general case there may not be such an element, but we consider only such $L$ for which it exists (Fig. \ref{gen_k7}).

Further, let's define a subset consisting of elements of the group using constraints on the degrees of the generators:

\begin{equation*}
    M_1 = M_1(L; t) = \{ 0 \} \cup \left\{ \phi_0^{\alpha_0} \phi_1^{\alpha_1} \dots \phi_p^{\alpha_p}    \right\}; 
 \end{equation*}   
    
\begin{equation*}
     0 \leq \alpha_0 \leq k-1; \; -t \leq \alpha_i \leq t.
\end{equation*}

Next, define

\begin{equation*}
    \operatorname{clip}(M; r) = \{ z \in M \; : \; |z|\leq r \}, \quad M \subset \mathbb{C}, \; r>0;
\end{equation*}

\begin{equation}
    M_s = \operatorname{clip}(M_{s-1} + M_1; r_s). 
    \label{minksum}
\end{equation}

Here we mean the Minkowski sum as sum of sets. Clipping parameters $r_2, \dots, r_s$ allows us to minimize the number of  vertices of small degree in the resulting graph. 

Thus, $M_1$ is a set having k-fold symmetry and consisting of the origin and a finite number of points on the circle $|z|=1$, $M_2=M_1+M_1$ is the Minkowski sum, etc. We can expect that the distance graph constructed on the set of vertices $M_s$ at sufficiently large $t$ and $s$ will contain a lot of copies of the original graph $L$ as subgraphs.  

\begin{figure}
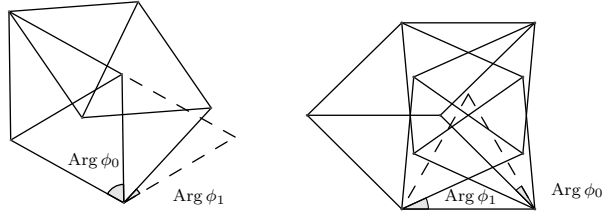

    \centering
\begin{tabular}{p{3.5cm} p{3.5cm} } 

\includegraphics[width=3cm]{gen_moser} &\includegraphics[width=4cm]{gen_sqrt2}
 \\

\end{tabular}    
    \caption{Generators of the group,  $L_7$ and $L_{10,2}$}
    \label{gen_k7}
\end{figure}

The last step of constructing the set of vertices is as follows:
\begin{equation*}
    W = M_s \cup \psi M_s,
\end{equation*}
where  $\psi \in \mathbb{C}$, $|\psi|=1$, defining the rotation, is chosen so that the distance graph built on the set of vertices $W$ contains edges connecting $M_s \setminus \{0\}$ and $\psi M_s \setminus \{0\}$  (Fig. \ref{fig_rotation}). Trying a finite set of  values for $\psi$, with proper choice of initial graph $L$ and other parameters we can discover a 5-chromatic distance graph. In case it is impossible to consider all variants, it seems a natural idea to arrange the found values of $\psi$ in descending order by the number of new edges.


\begin{figure}
    \centering
    \includegraphics[width=4cm]{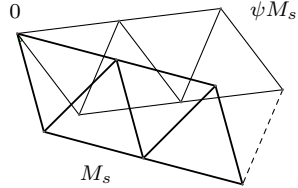}
    \caption{Rotation about the origin }
    \label{fig_rotation}
\end{figure}

\begin{rmk} Recall that the radius of a graph $G$ is  
$$r(G) = \min_{u \in V(G)} \max_{v \in V(G)} \operatorname{d} (u,v).$$
Here $\operatorname{d}(u,v)$ is the distance in the abstract graph. Then the minimal $s$ at which the construction described above contains a subgraph isomorphic to $L$ is not smaller than  $r(L)$.  Therefore, it is preferable to consider graphs with a smaller radius. For two graphs considered in the following section as a basic element of the construction $r(L_7)=r(L_{10,2})=2$. 
\end{rmk}

\begin{rmk}
 If we add to the generators of a discrete group an element that translates one vertex of some unit segment to the other, it will not be necessary to determine the Minkowski sum.  One can define the set of vertices as a subset of an orbit of some point under the action of a discrete group of isometries subject to limitations on the degrees of the generators. Such a construction can be applied to a distance graph embedded in a sphere or in the hyperbolic plane. 
\end{rmk}

\section{New examples for the plane}

This section presents the results of two series of experiments. In the first one, the Moser spindle was used as graph $L$, and in the second one, the graph $L_{10,2}$ was used. 

In both cases the removal of part of the vertices was applied only after the first addition as in the paper by M. Heule \cite{Heule1}, i.e. clipping parameters in (\ref{minksum}) are $r_2 = 1$, $r_3 = \infty$.

\textbf{Series 1}. Let $L = L_{7}$. Then one of the possible set of generators is 
$$\phi_0 = \frac{1}{2}+ \frac{\sqrt{3}}{2}i,\quad \phi_0^6=1, \quad \phi_1 = \frac{\sqrt{33}}{6}+ \frac{\sqrt{3}}{6}i.$$ 
Choose $t=2$. Then the set $M_1(L_{7}, t)$ consists of 31 points, including the zero, and the set $M_3$ contains 1939 points, following \cite{Heule1}.  

\textbf{Series 2}. Let $L = L_{10, 2}$. Choose the generators: $$\phi_0 = \frac{\sqrt{6}+\sqrt{2}}{4}+ \frac{\sqrt{6}-\sqrt{2}}{4}i,\quad \phi_0^{24}=1, \quad \phi_1 =  \frac{\sqrt{6}}{3}+\frac{\sqrt{3}}{3}i.$$ Let $t = 1$. Then $M_1(L_{10,2}, t)$ contains $73$ points, and $M_3$ contains
32257 points.  

The results are given in Table~\ref{tab_series}. By \emph{common edges} we mean edges of $\operatorname{udg}(M_3 \cup \psi M_3)$ that exist for any $\psi$ of unit norm. Accordingly, the \emph{new edges} are of the form $(u,v)$, $u \in M_3 \setminus \{0\}$, $v \in \psi M_3 \setminus \{0\}$.

\begin{rmk}
In the second series only the cases where some new edge appears between sets $M_2 \setminus \{0\}$ and $\psi M_2 \setminus \{0\}$ were fully considered.
\end{rmk}

\begin{table}[ht]
    \centering
    \begin{tabular}{|c|c|c|c|c|c|c|}
    \hline
        \# & Graph $L$ & Vertices  &  Common edges  &  Cases & 5-chromatic   \\
        \hline
        1 & $L_{7}$ & 3877 &  26748 & 31375 & 5 \\
        \hline
        2 & $L_{10,2}$ & 64513 & 542352  & 2731 & 14  \\
        \hline
    \end{tabular}
    \caption{Series of computations.}
    \label{tab_series}
\end{table}

In the first series of experiments, in addition to the  construction found by M. Heule (graph \#1, for which  $\psi = \psi_*=\frac{7}{8} + \frac{\sqrt{15}}{8}i$),  other similar cases were found. All the 5-chromatic distance graph found can be embedded in $\mathbb{Q}(i,\sqrt{3},\sqrt{5},\sqrt{11})\subset\mathbb{C}$. The values of $\psi$ (i.e. rotation angles found) are shown in Table \ref{tab_results1}. Note that the first 5-chromatic graph appears as number 6 in the enumeration, and it takes less than 1 minute to find it.

\begin{table}[ht]
    \centering
    \begin{tabular}{|c|c|c|c|}
    \hline
         \# & $\operatorname{Im}\psi$ & $\psi$ & New edges \\
        \hline
        $1$ & $0.48412291...$ & $\psi_*$ & 66\\
        \hline
        
        $2$ & $0.24642936...$ & $\phi_0 \phi_1^{-1} \psi_*^{-1}$ & 48 \\
        \hline
        $3$ & $0.21092166...$ & $ \phi_1^{-1} \psi_*$ & 48 \\
        \hline
        $4$ & $0.08023868...$ & $ \phi_1^{2} \psi_*^{-1}$ & 30 \\
        \hline
        $5$ & $0.04383445...$ & $\phi_0^{-1} \phi_1^{2} \psi_*$ & 30 \\
        \hline
    \end{tabular}
    \caption{5-chromatic graphs, series 1}
    \label{tab_results1}
\end{table}

In the second series of experiments, there were new constructions of 5-chromatic unit-distance graphs (Table~\ref{ang_sqrt2}, Fig.~\ref{patterns2}). Except for the three cases, we can assume $z=2$ and $w \in M_2$  choosing a pair of vertices $(z,w)$ that determines the equation for $\psi' = \psi\phi_0^p$ 
\[
|z-\psi' w|=1, 
\]
see Fig. \ref{rotate_sqrt2}. 
\begin{table}[ht]
    \centering
    \begin{tabular}{|c|c|c|c|}
    \hline
        \# & $\operatorname{Im}\psi$ & $(z,w)$ & New edges   \\
        \hline
        $1$ & $0.12388487...$ & $2, 1+\phi_0^5$ & 240 \\
        \hline
        $2$ & $0.11455695...$ & $2, 1+\phi_0$ & 240 \\
        \hline
        $3$ & $0.00903168...$ & $2, 1+\phi_0^7$ & 240 \\
        \hline
        $4$ & $0.08275438...$ & $2, \phi_1(1+\phi_0^7)$ & 144  \\
        \hline
        $5$ & $0.03231741...$ & $2, \phi_1(1+\phi_0^5)$ & 144 \\
        \hline
        $6$ & $0.02292606...$ & $2, \phi_1(1+\phi_0)$ & 144 \\
        \hline
        $7$ & $0.05508155...$ & $2, \phi_1^{-1}(1+\phi_0)$ & 144\\
        \hline
        $8$ & $0.04569855...$ & $2, \phi_1^{-1}(1+\phi_0^5)$ & 144\\
        \hline
        $9$ & $0.01007415...$ & $2, \phi_1^{-1}(1+\phi_0^7)$ & 144\\
        \hline
        $10$ & $0.01823725...$ & $2,2$ & 120 \\
        \hline
        $11$ & $0.08672634...$ & $\phi_1\phi_0^{5}+\phi_1^{-1}\phi_0^{-5}, \phi_1\phi_0^{2}+\phi_1^{-1}\phi_0^{-2}$ & 120\\
        \hline
        $12$ & $0.12702917...$ & $\phi_1\phi_0^{4}+\phi_1^{-1}\phi_0^{-4}, \phi_1\phi_0^{2}+\phi_1^{-1}\phi_0^{-2}$ & 48\\
        \hline
        $13$ & $0.00301228...$ & $2, \phi_1 \phi_0^4 + \phi_1^{-1} \phi_0^{-4}$ & 48 \\
        \hline
        $14$ & $0.09860920...$ & $\phi_1\phi_0^{4}+\phi_1^{-1}\phi_0^{-4}, \phi_1\phi_0^{4}+\phi_1^{-1}\phi_0^{-4}$ & 24 \\
        \hline
    \end{tabular}
    \caption{5-chromatic graphs, series 2}
    \label{ang_sqrt2}
\end{table}

\begin{prop}
\[
\chi\left(\mathbb{Q}(i,\sqrt{2},\sqrt{3},\sqrt{5})\right)\geq 5.
\]
\end{prop}

\begin{figure}[ht]
    \centering
\includegraphics[width=7cm]{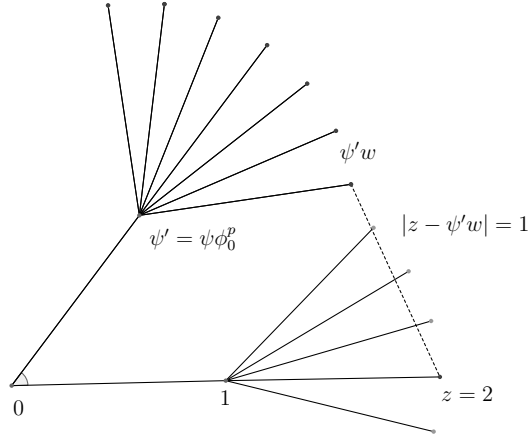}
    \caption{The equation for $\psi$, cases 1-3, series 2}
    \label{rotate_sqrt2}
\end{figure}

\begin{proof}
In fact, in case 10, 
\[
\psi = \phi_0^{-1}\left(\frac{7}{8} + \frac{\sqrt{15}}{8}i\right). 
\]
Therefore, the graph constructed in the case 10 can be embedded in  
\[
\mathbb{Q}(i,\sqrt{2},\sqrt{3},\sqrt{5}).
\] 
\end{proof}

As a simple corollary one can observe that this graph can't contain the Moser spindle as a subgraph because this field does not contain $\sqrt{11}$. This statement is also true for the other graphs constructed in the second series of computations.

\begin{prop}
    None of the 5-chromatic graphs $G_{64513,k}$, $k=1, \dots, 14$ contain $L_7$ as a subgraph.
\end{prop}

\begin{proof}
Indeed,  $M_3 \subset V(G_{64513,k})$ contains at most 4 vertices of $L_7$ forming a diamond subgraph (i.e. $D_4 = K_4 \setminus e$). So, $L_7$ can be a subgraph of $G_{64513,k}$ only if two diamond subgraphs of $L_7$ belong to two subsets $M_3$ and $\psi M_3$ with a single common vertex. On the other hand, it's easy to check that the values of $\psi$ from Table \ref{ang_sqrt2} do not coincide with $\phi_0^q\phi_1^p\left(\frac{5}{6} \pm\frac{\sqrt{11}}{6}i\right)$, $0 \leq q \leq 23$, $-2\leq p \leq 2$ which is necessary to contain $L_7$ as a subgraph.
\end{proof}

Analytic verification of existence of the graphs constructed in this section is done in Sage 9.2 environment, and takes 1-2 minutes for a graph with 64513 vertices, if only new edges are checked~\cite{Github}.

\begin{figure}[ht]
\begin{center}
\begin{tabular}{ p{3cm} p{3cm} p{3cm}  } 
\includegraphics[width=3cm]{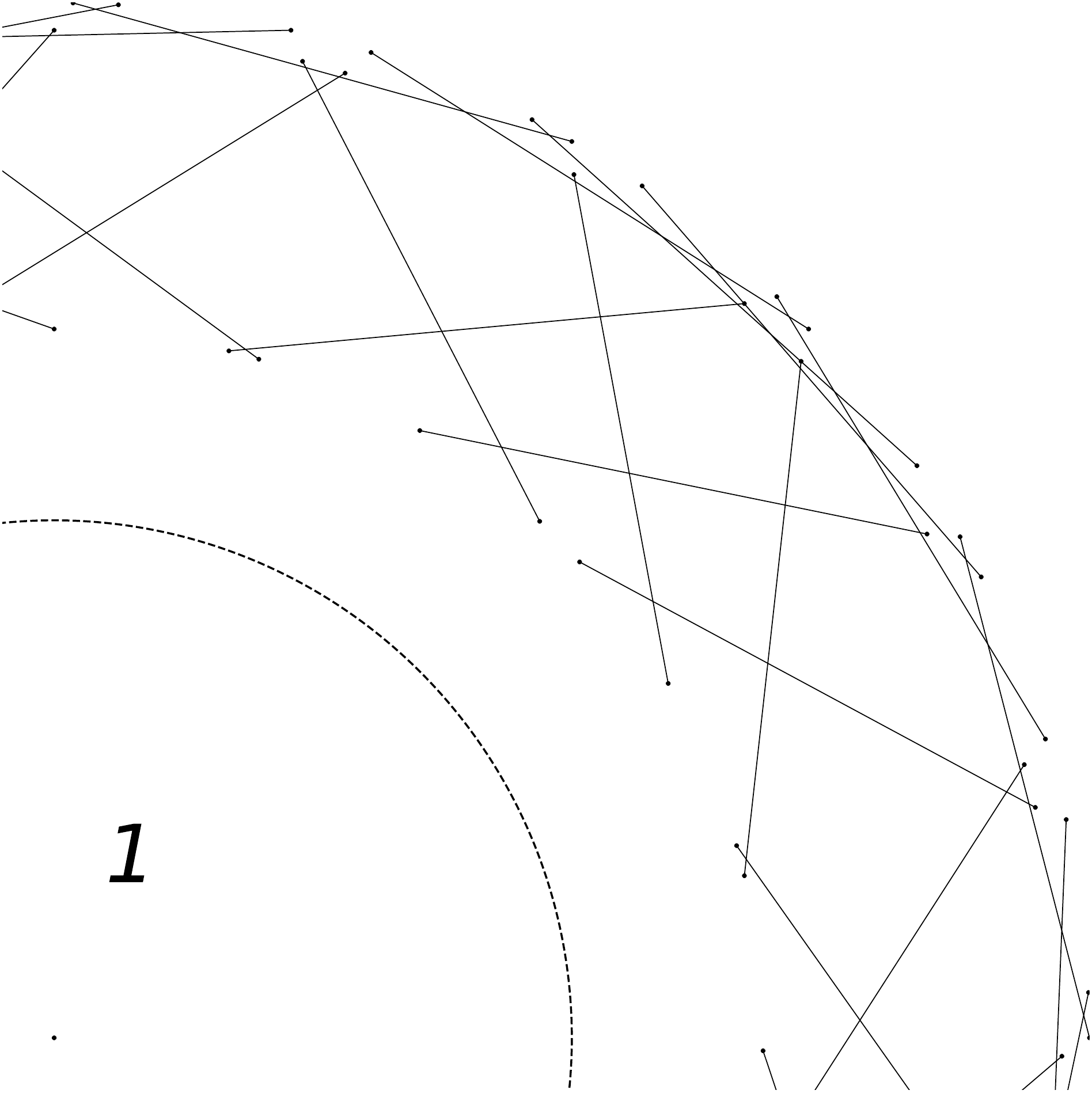} & \includegraphics[width=3cm]{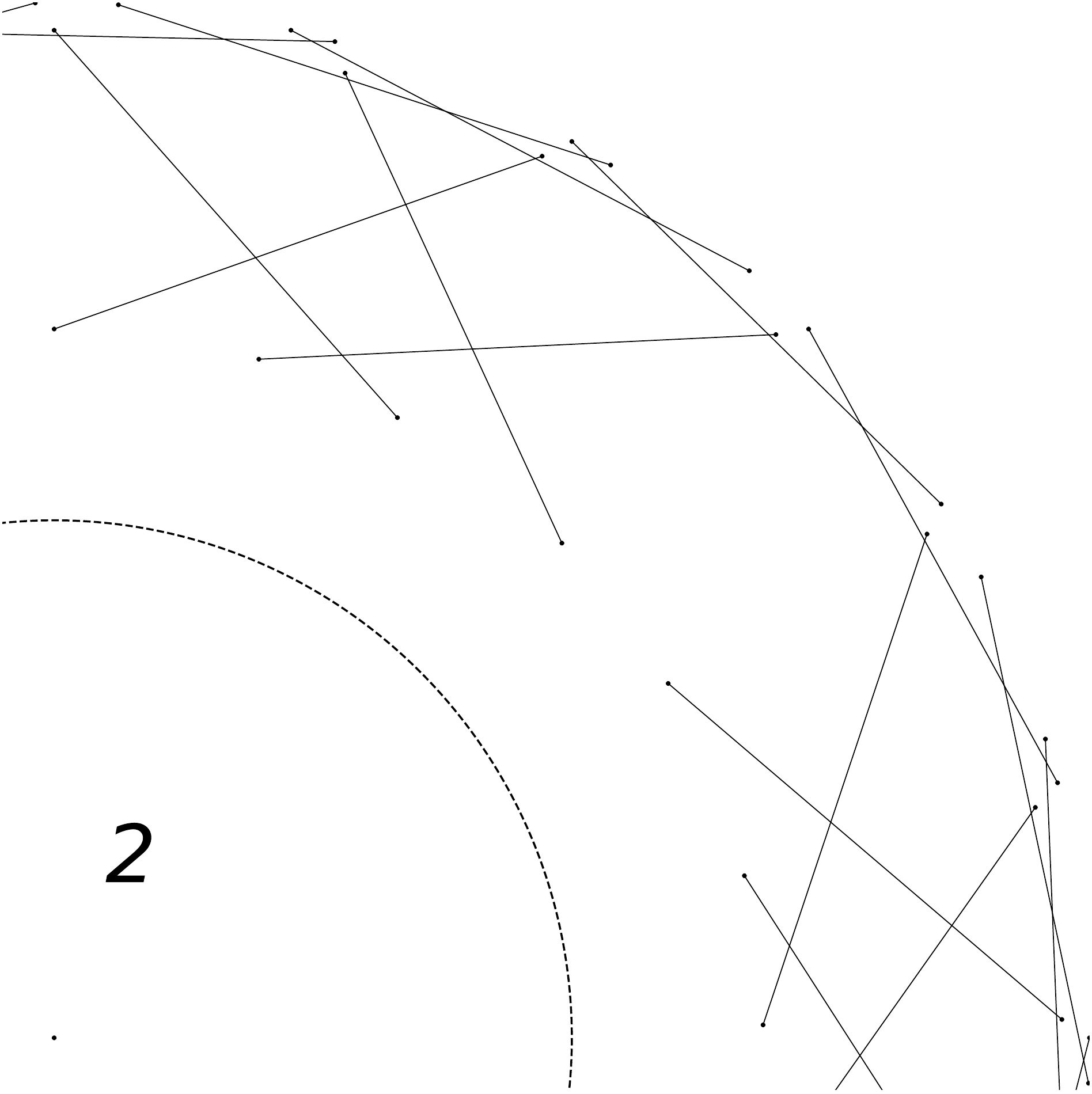} & \includegraphics[width=3cm]{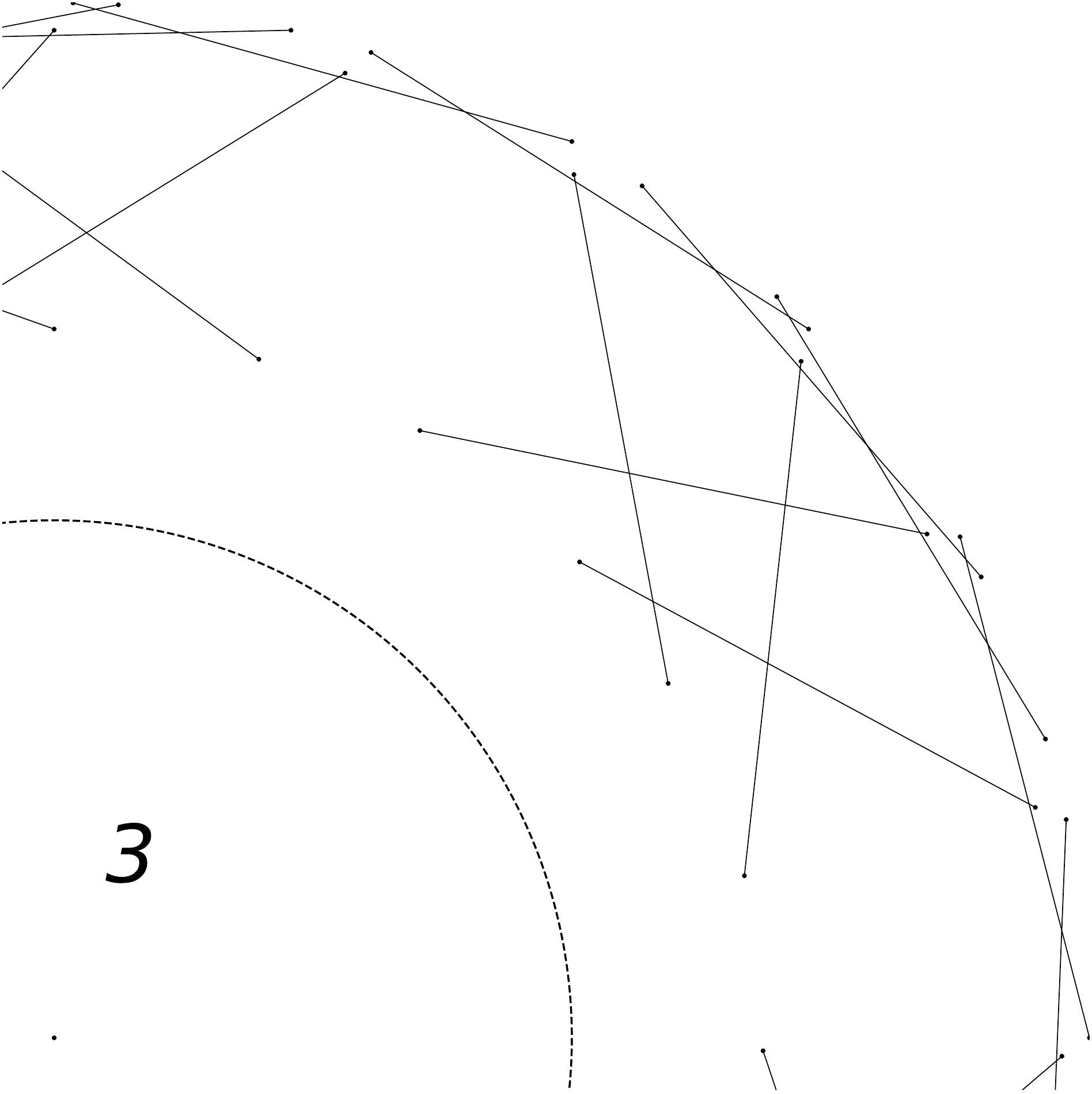} \\ \includegraphics[width=3cm]{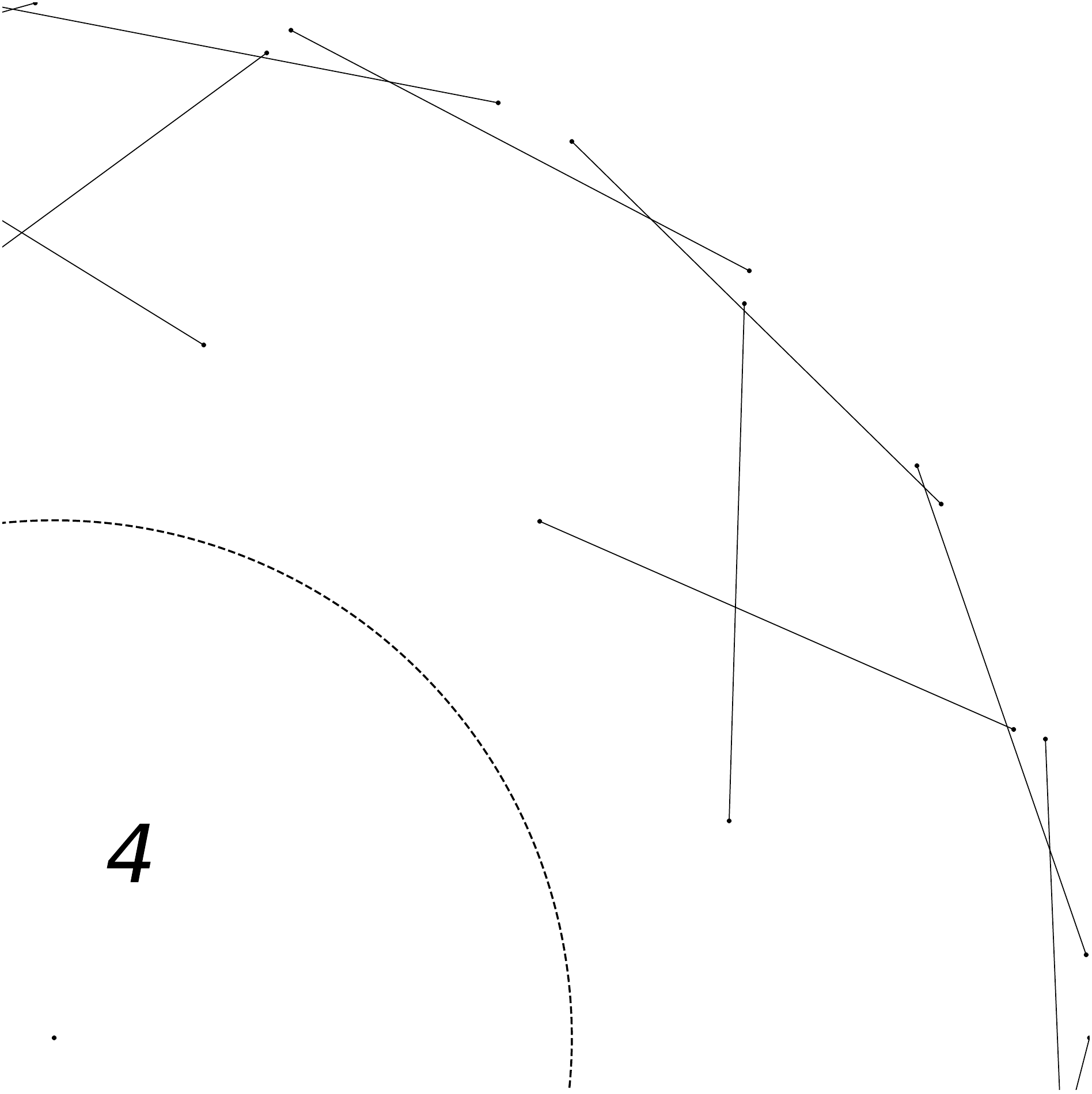} & \includegraphics[width=3cm]{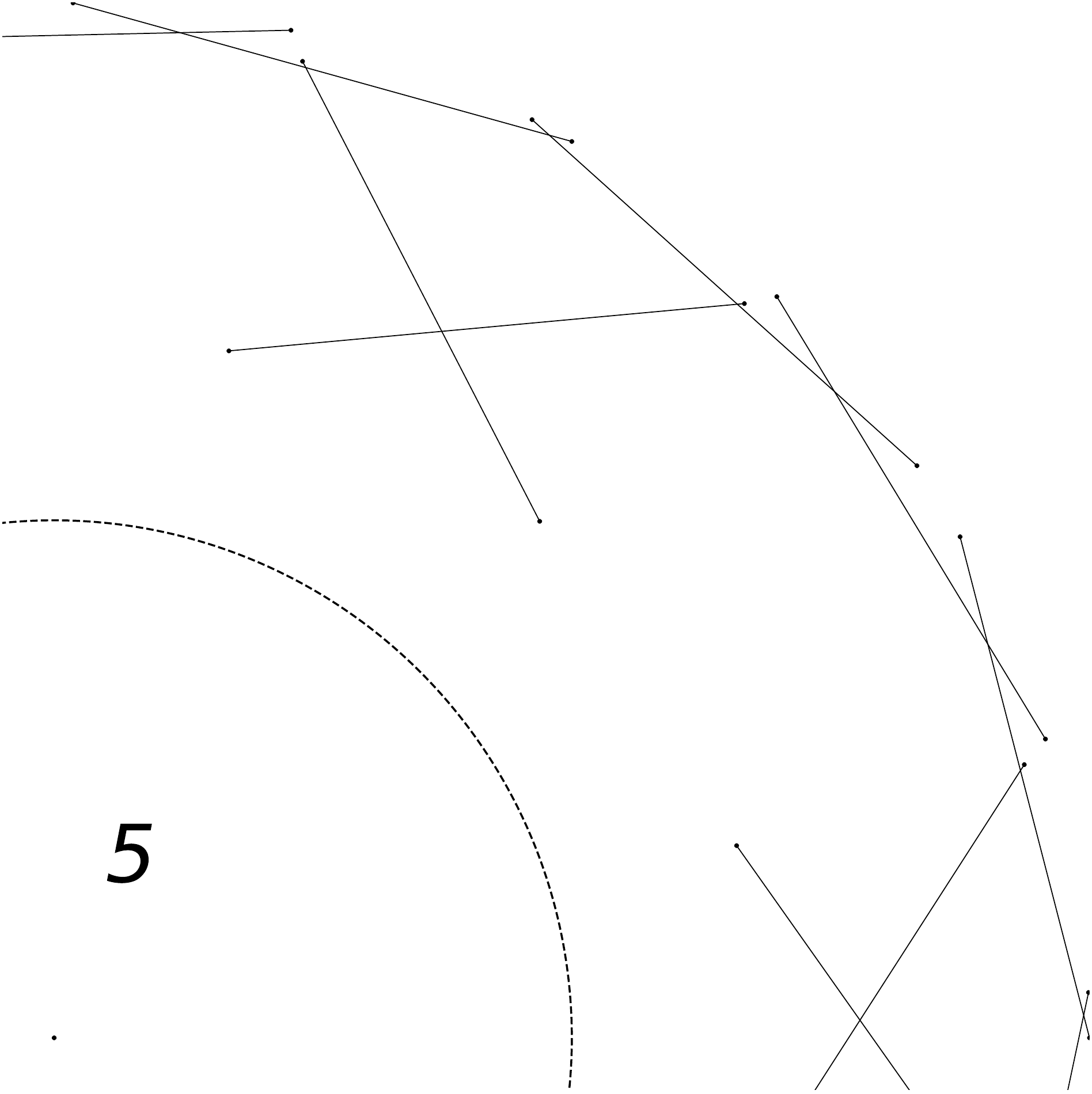} &
 \\
\end{tabular}
\end{center}
\caption{New edges in 5-chromatic graphs, series 1 }
\label{patterns1}
\end{figure}

\begin{figure}[ht]
\begin{center}
\begin{tabular}{ p{2.5cm} p{2.5cm} p{2.5cm} p{2.5cm} } 
\includegraphics[width=2.5cm]{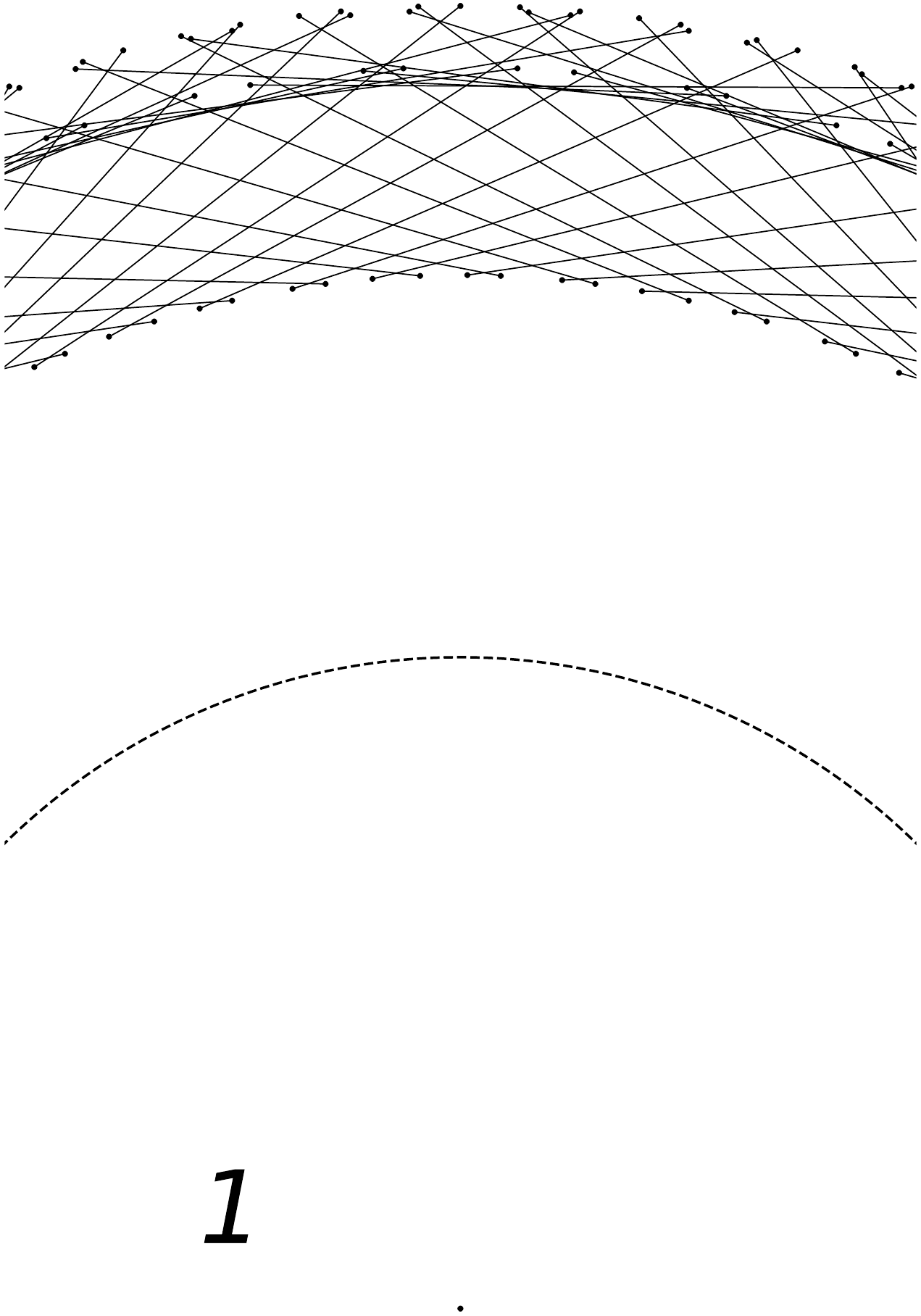} & \includegraphics[width=2.5cm]{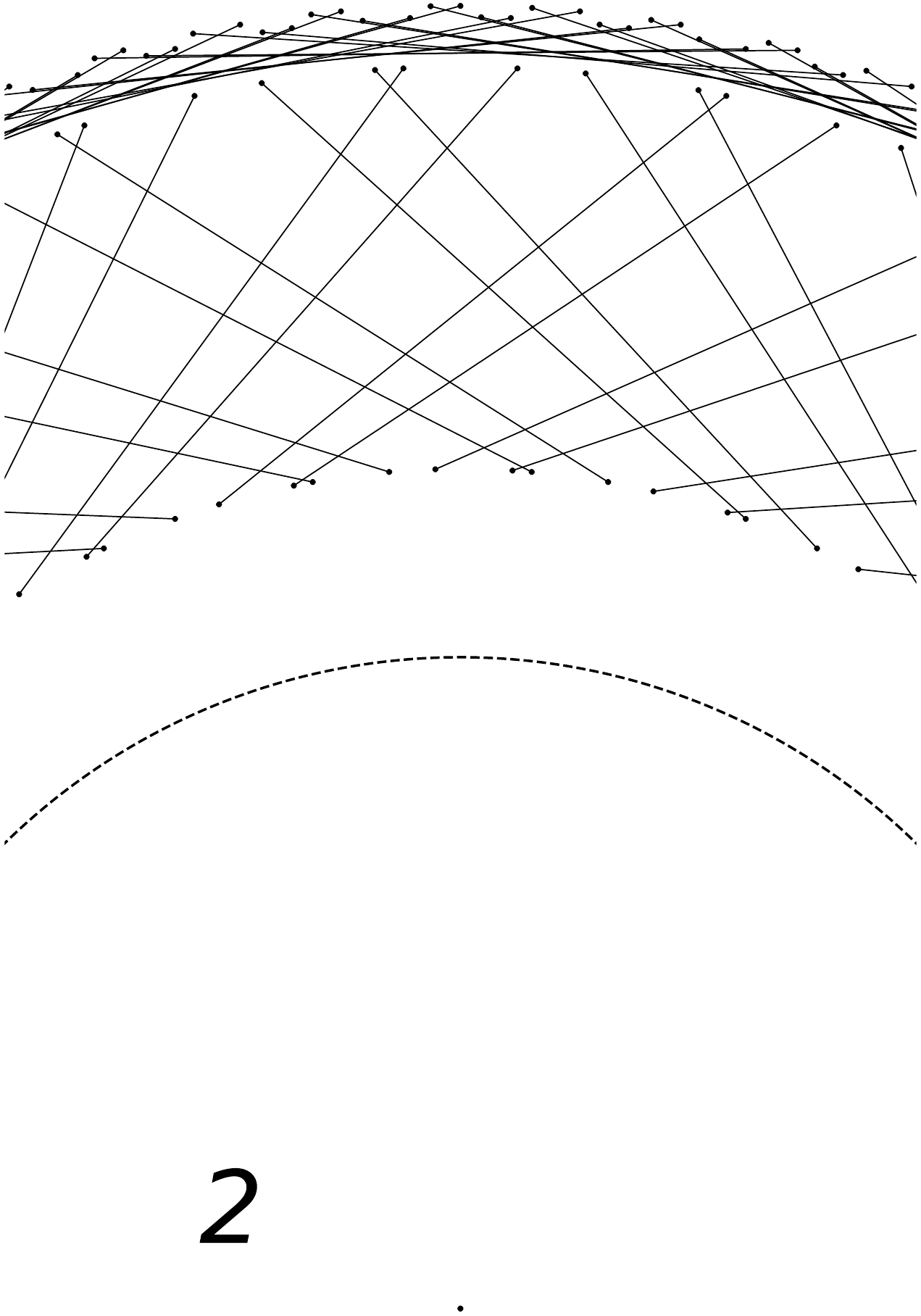} & \includegraphics[width=2.5cm]{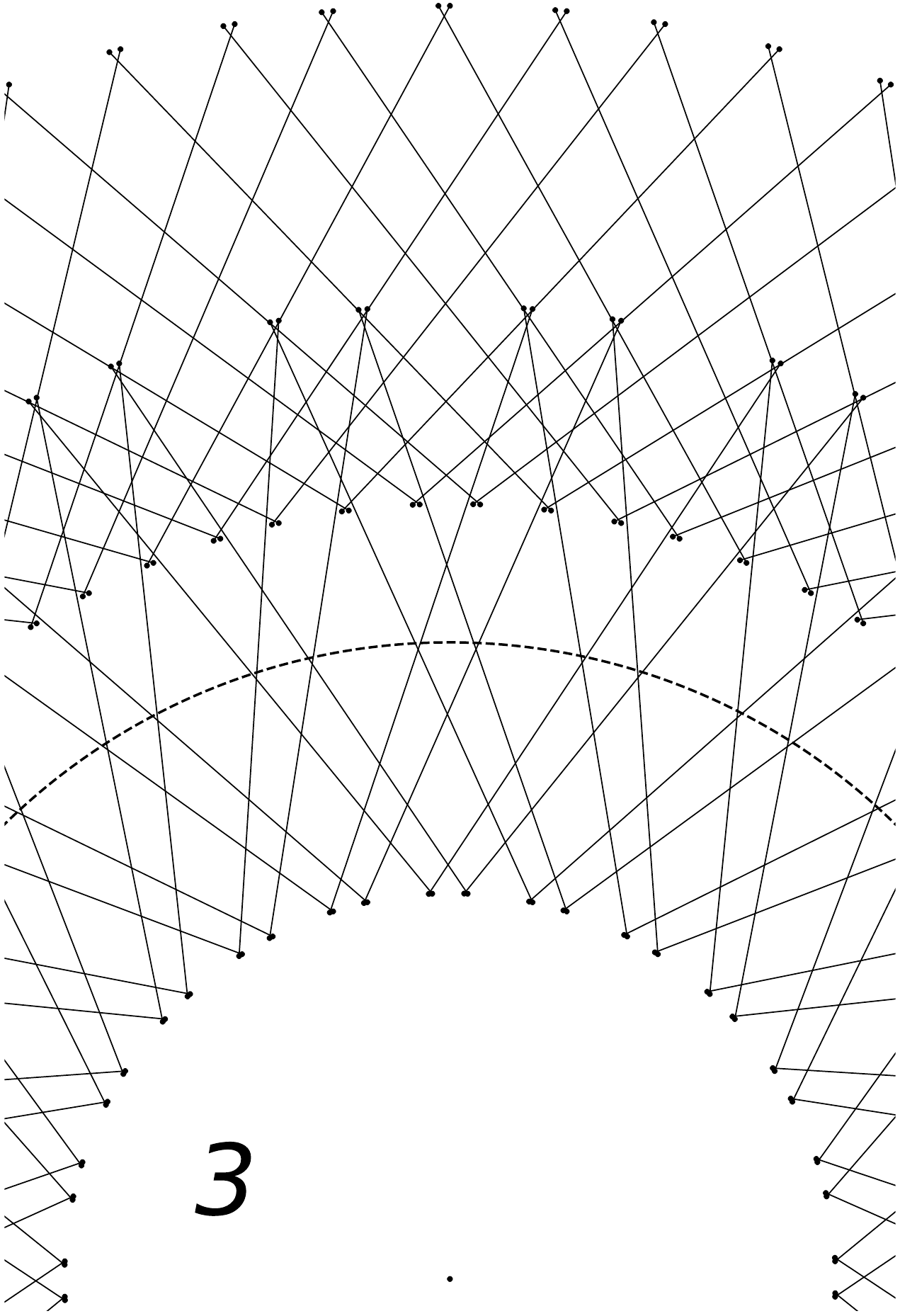} & \includegraphics[width=2.5cm]{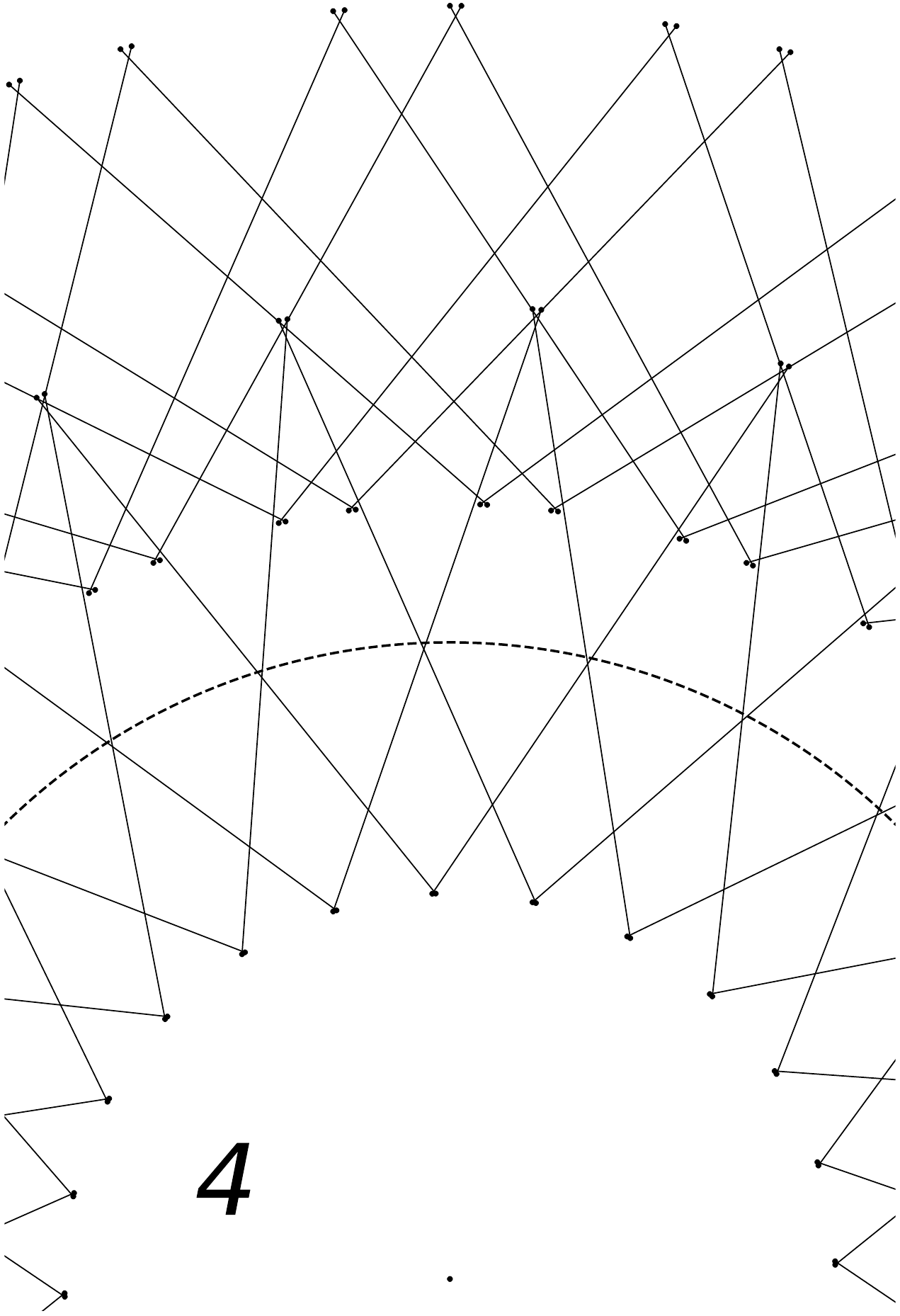} \\ \includegraphics[width=2.5cm]{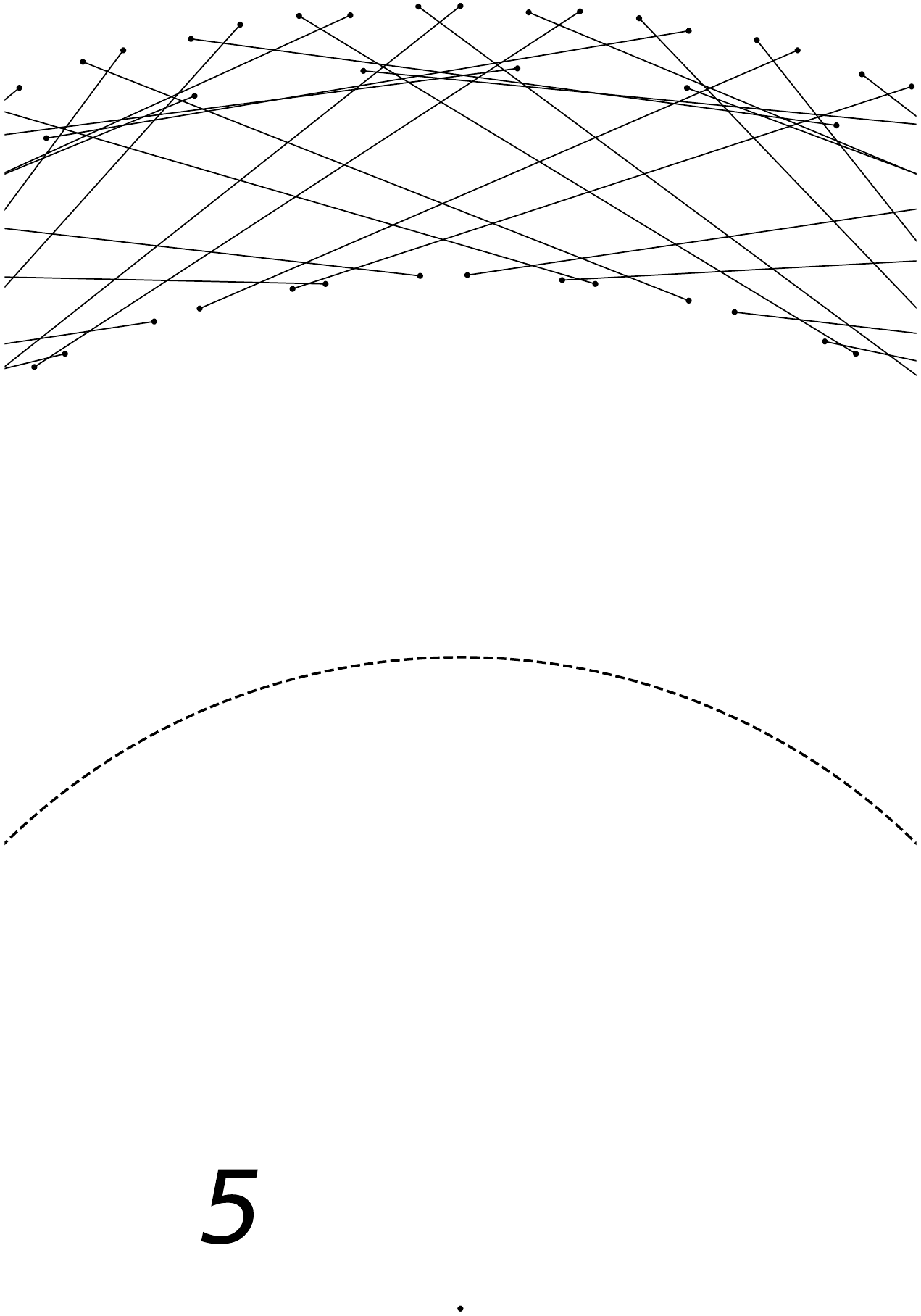}  &
 \includegraphics[width=2.5cm]{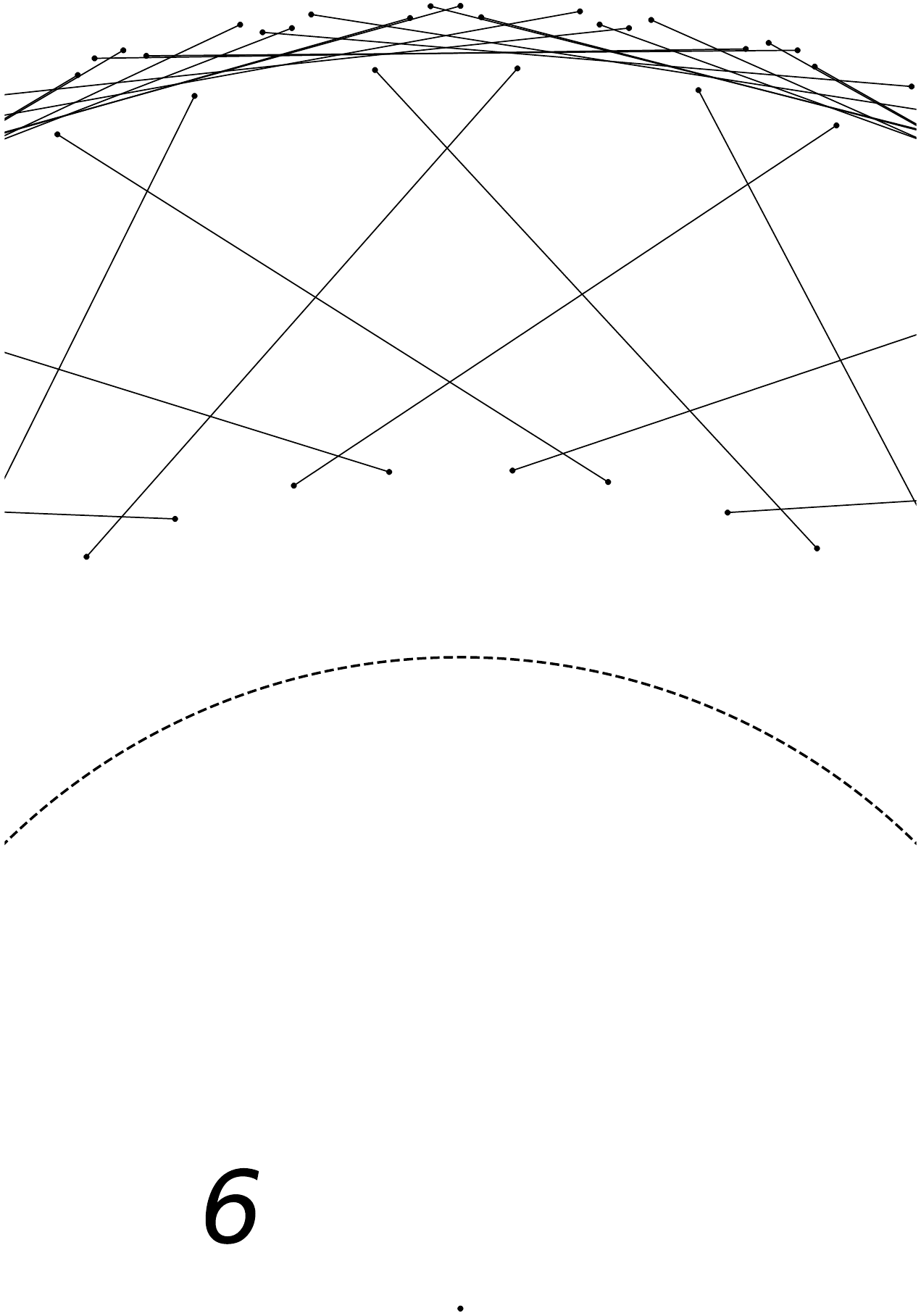} & \includegraphics[width=2.5cm]{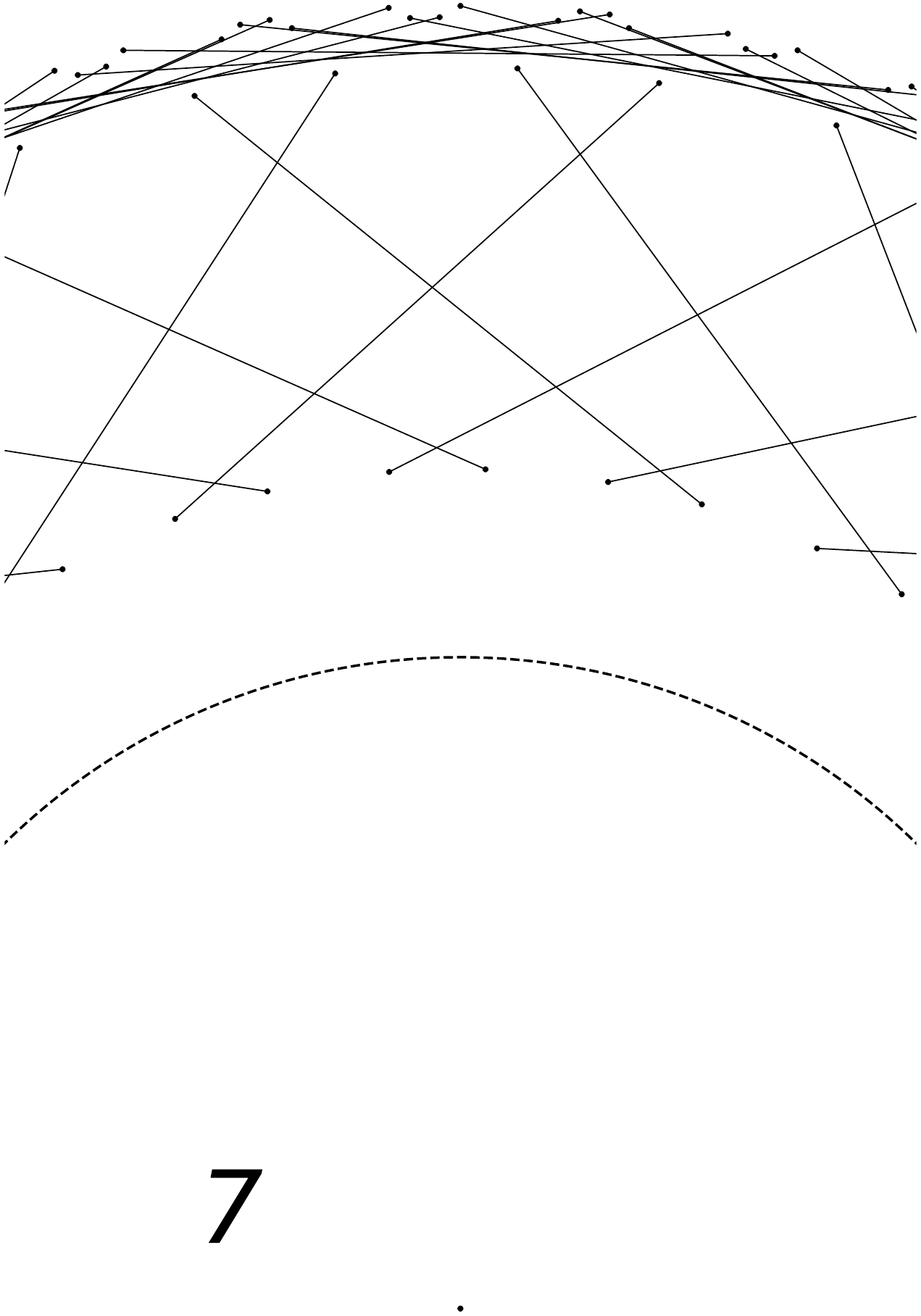} & \includegraphics[width=2.5cm]{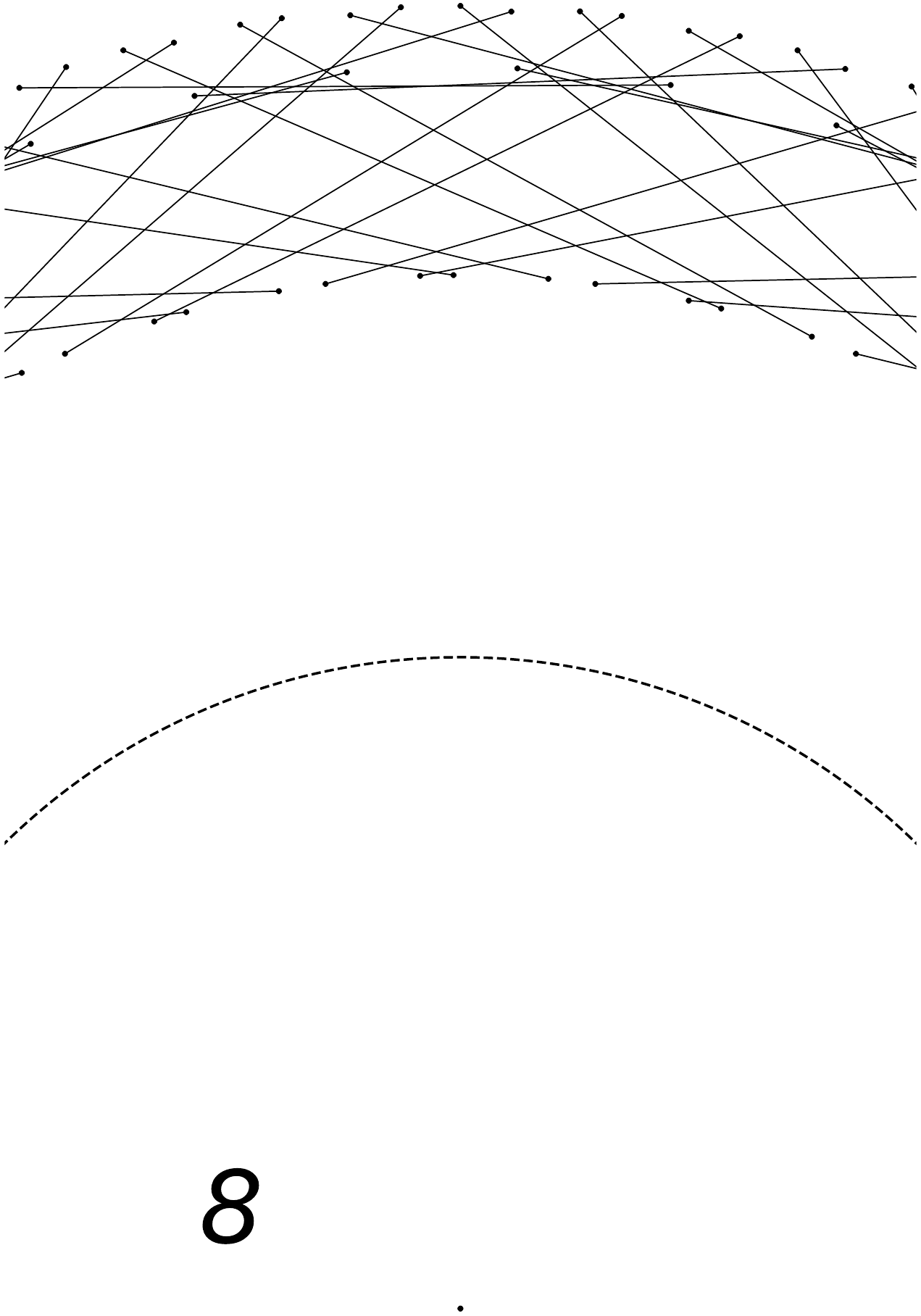} \\ \includegraphics[width=2.5cm]{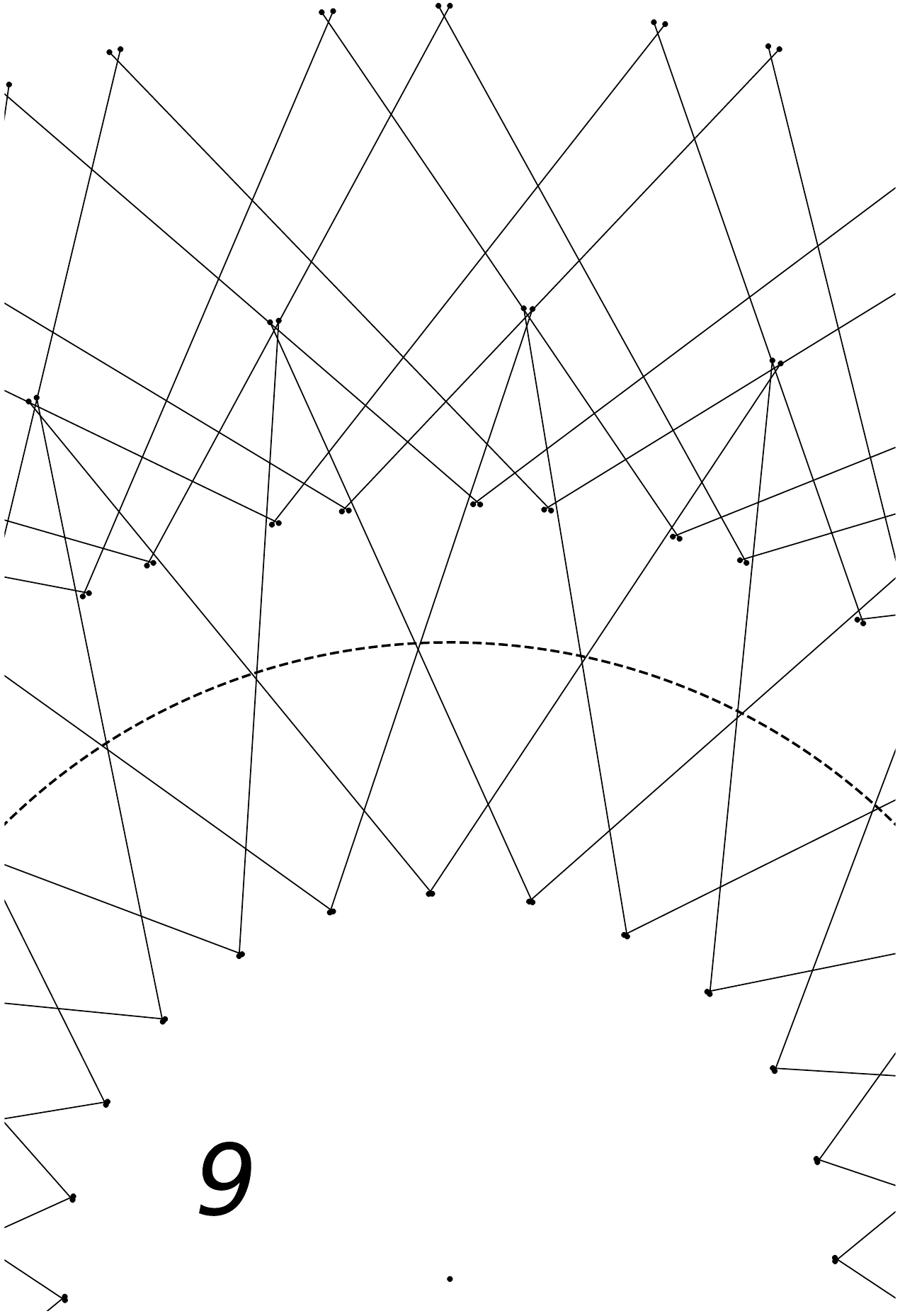} & \includegraphics[width=2.5cm]{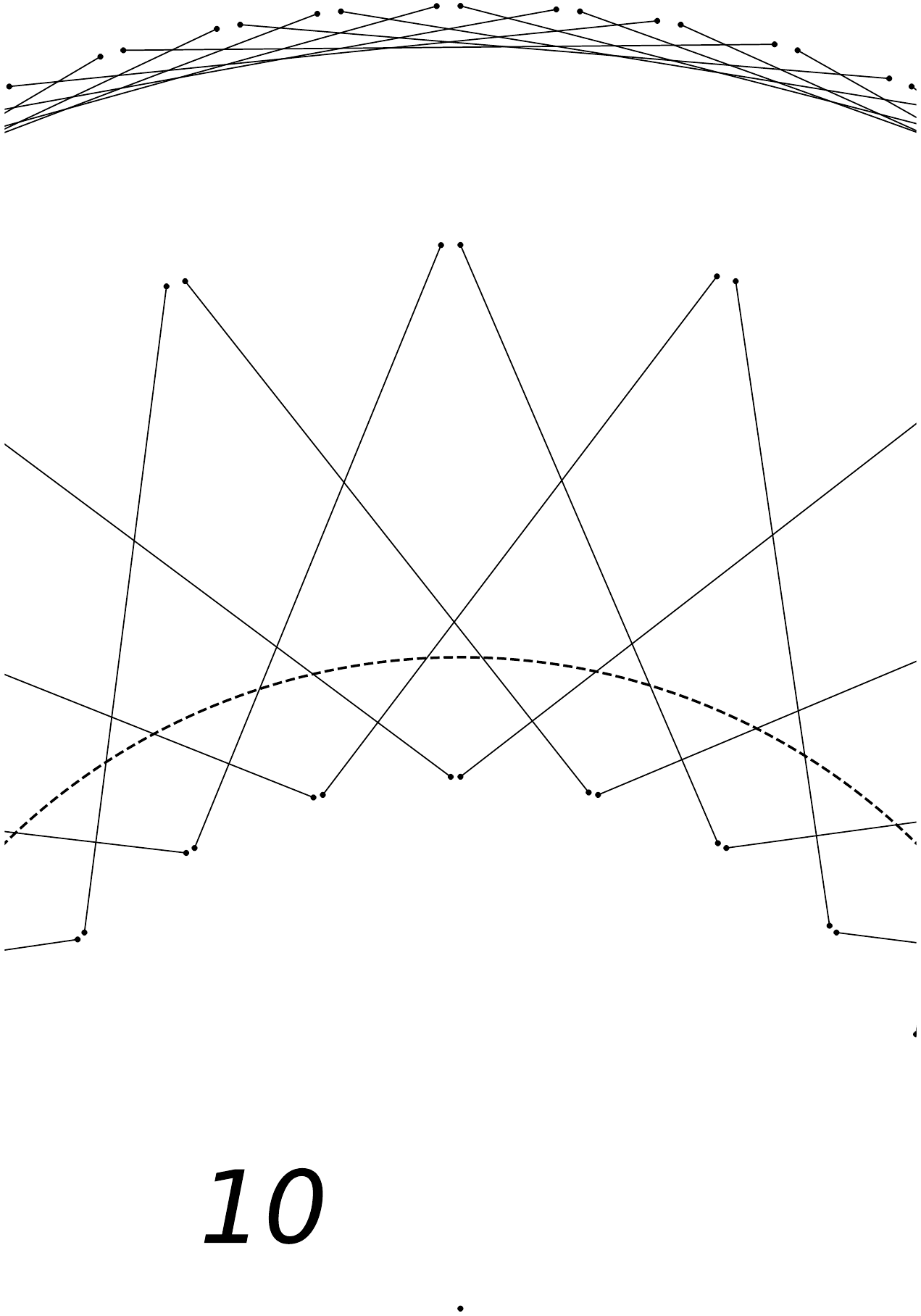} &
 \includegraphics[width=2.5cm]{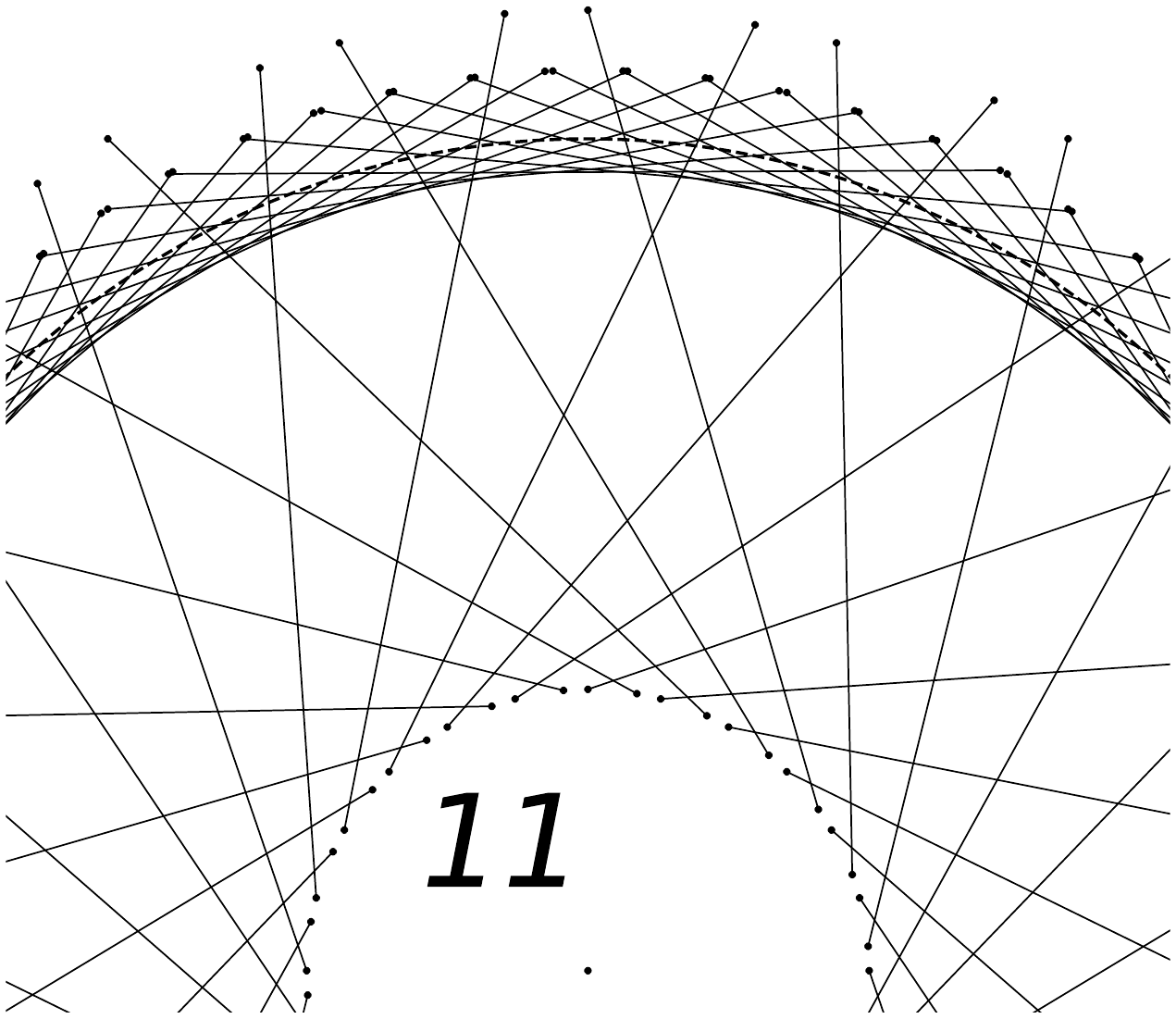} & \includegraphics[width=2.5cm]{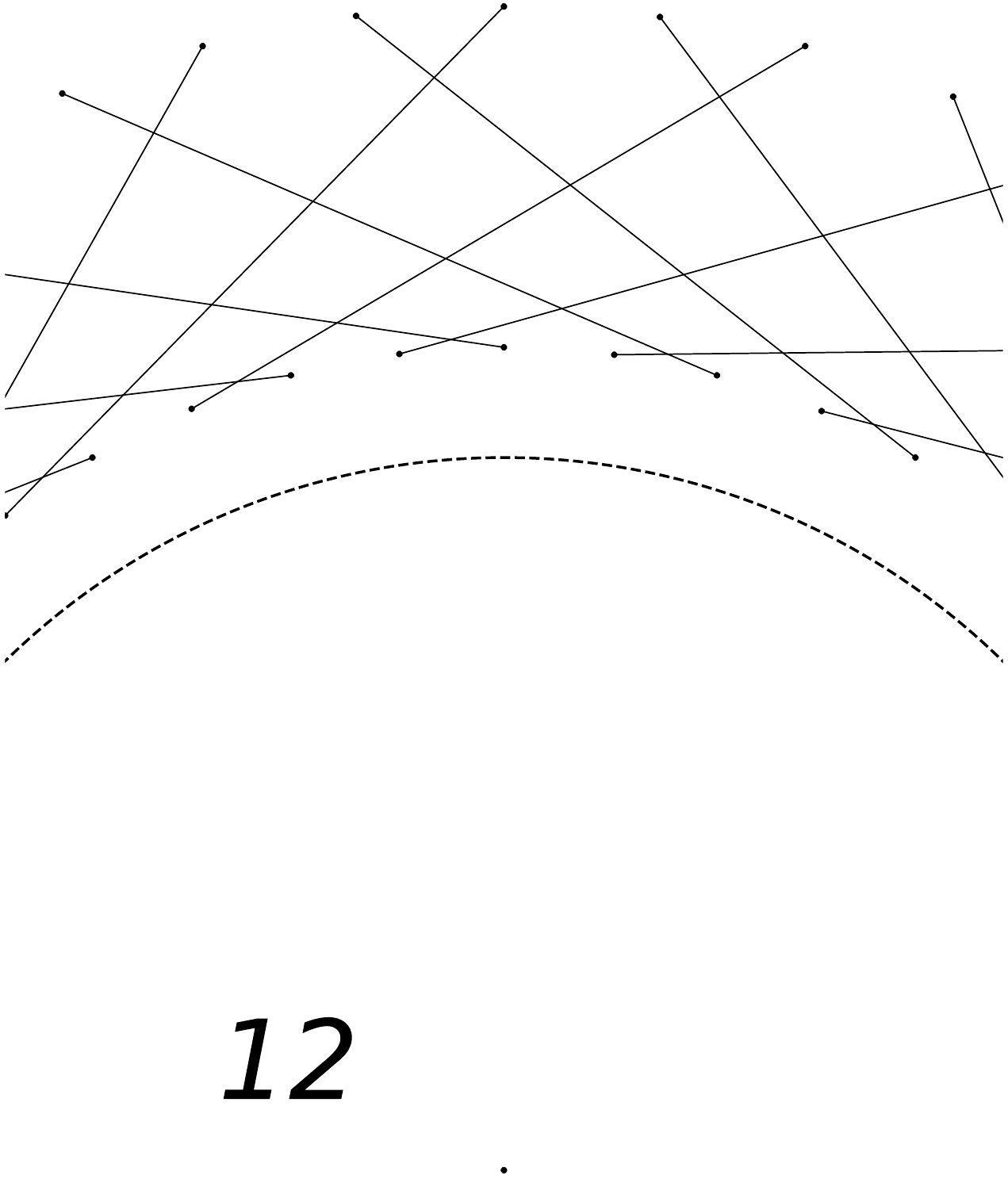} \\ \includegraphics[width=2.5cm]{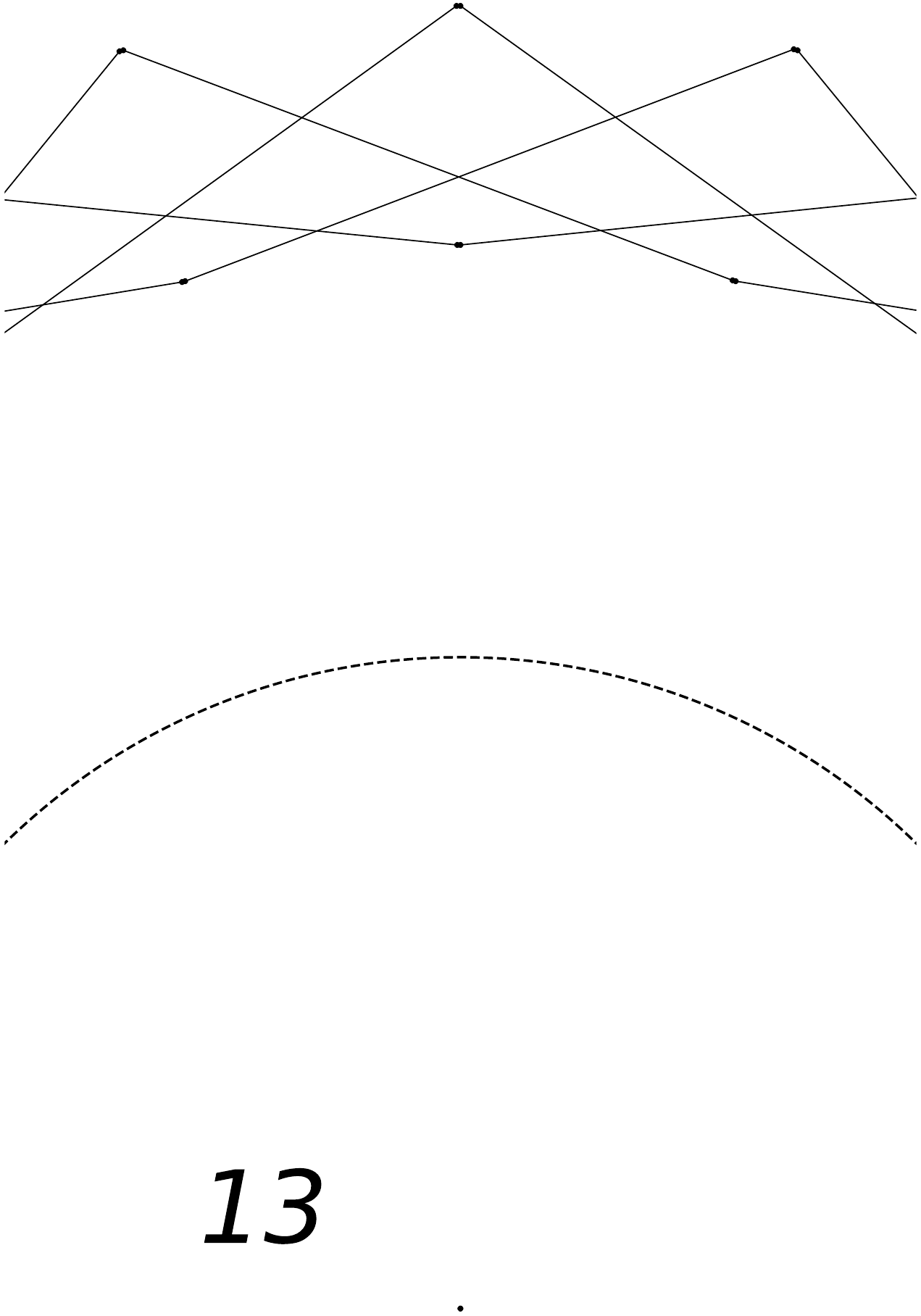} & \includegraphics[width=2.5cm]{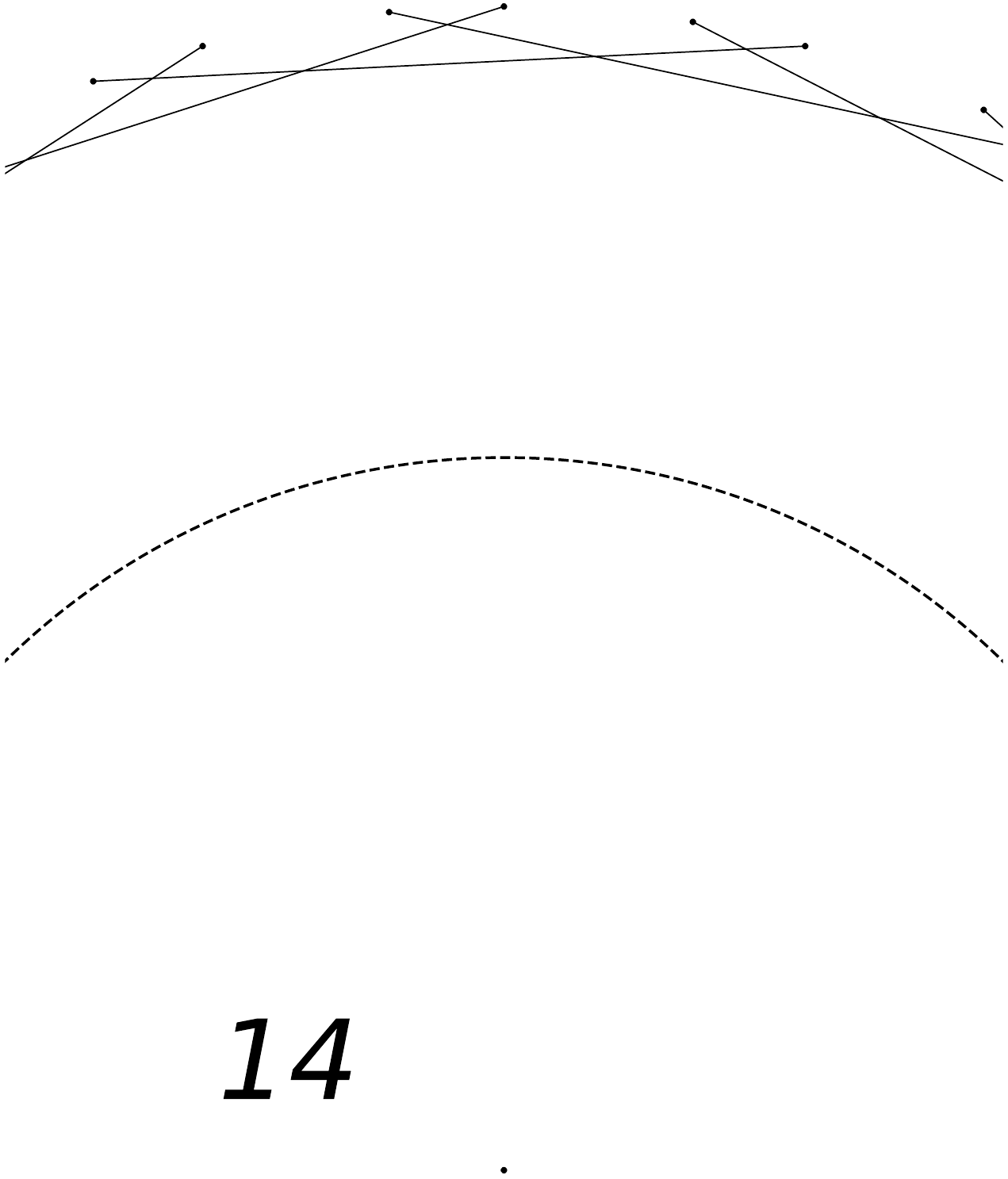} & & 
 \\
\end{tabular}
\end{center}
\caption{New edges in 5-chromatic graphs, series 2}
\label{patterns2}
\end{figure}

\section{Performance of SAT solvers}

When checking the existence of coloring of the constructed graphs in 4 colors,  SAT solvers \textbf{Minisat}\cite{Minisat}, \textbf{Glucose 4} and the multithreaded version \textbf{Glucose-syrup} \cite{Glucose}, \textbf{Maple\-LCMDist\-Chrono\-BT-DLv3} (MapleLCMDv3) \cite{Zaikin}, \textbf{plingeling} \cite{Biere1}, \textbf{painless-mcomsps} \cite{painless}, and  \textbf{Kissat} \cite{Biere2} were used. All employed SAT solvers are based on the Conflict-Driven Clause Learning algorithm (CDCL \cite{CDCL}).

For the sake of completeness let us describe a Boolean satisfiability  (SAT) problem corresponding to the problem of finding a proper graph coloring with k colors. Let $x_{ic} \in \{0,1\}$ means ``vertex $u_i$ has color $c$''.  Then the conjunctive normal form (CNF) of the Boolean formula is as follows
\begin{equation}
\begin{split}
    f_G(X) =  \left(\bigwedge_{1\leq i \leq n} \bigvee_{1\leq c\leq k} x_{ic} \right)  
    \wedge \left(\bigwedge_{(u_i,u_j) \in E(G),\; 1 \leq c \leq k} \overline{ x_{ic}} \vee \overline{ x_{jc}} \right). 
\end{split}   
\label{cnf}
\end{equation}

In all problems of this type it is useful to assume that at some maximal clique the vertex colors are already defined (``breaking symmetry'').

It turns out that in the vast majority of cases the SAT problem is quite simple, which allows us to perform such a large scale enumeration.  The Table~\ref{sat_perf} shows the runtimes of the problems described in the previous section. 
\begin{enumerate}
    \item Test 1: 5 graphs, 3877 vertices,  no 4-coloring.
    \item Test 2: 100 graphs, 64513 vertices, 4-coloring exists.
    \item Test 3: 5-chromatic graphs 1-11 (series 2), 64513 vertices.
\end{enumerate}
Problems for graphs with 3877 vertices (test 1) do not present any difficulties for SAT solvers. As for the second and third tests, the conclusions are quite simple. Namely, the \textbf{Glucose 4} solver is the most efficient among both single-threaded solvers and multithreaded ones. Of course, the standings can change if we adjust the solvers' settings or choose a different encoding.

In the experiments we used a PC with a  Core i9-7940X CPU (3.8--4.3 GHz, 14 cores, 28 threads), 64 Gb RAM, running Ubuntu 20.04. The multithreaded solvers were run with 24 threads.

\begin{table}
    \centering
    \begin{tabular}{|c|c|c|c|}
    \hline
    SAT solver & 3877, UNSAT & 64513, SAT & 64513, UNSAT  \\
    \hline
    Minisat   & 6.6 (2.9--12)  & 180 (5.4--1500) & 5200 (2800--8400)  \\
    \hline
    Glucose 4  & 4.8 (3.9--5.7) & 74 (4.8--1200) & 1400 (1100--1700) \\
    \hline
    MapleLCMDv3 & 4.3 (3.4--6.4) & 120 (24--650) & 2400 (1900--2900) \\
    \hline
    Kissat   & 6.4 (4.1--8.1) & 110 (0.53--1200) & 4500 (3200--6600) \\
    \hline
    \hline
    Glucose--syrup (24t) & 3.3 (3.0--3.6)  & 30 (10--260) & 640 (500--820)  \\
    \hline
    plingeling (24t) & 2.5 (1.8--3.8) & 48 (3--1000) & 3300 (1800--4300) \\
    \hline
    painless (24t) & 6.0 (5.2-6.3) & 56 (30--490) & 660 (530--920) \\
    \hline
    \end{tabular}
    \caption{Elapsed time, avg (min--max), seconds.}
    \label{sat_perf}
\end{table}

\section{Conclusion}

The results given in the paper show that 5-chromatic unit-distance graphs embedded in surfaces of constant curvature can be obtained, at least in some cases, by straightforward algorithms. Except for the latest SAT solvers, the software we used for the search was  weakly optimized and was at the level of prototypes. It is possible that optimization of algorithms and program code, as well as the use of large computational power may allow us to find many new examples of 5-chromatic graphs or even a 6-chromatic unit distance graph embedded in some  surface of constant curvature. 

So far we have not been able to check the existence of a 5-coloring for some subgraphs of $H_{12,1} \circ H_{9,1} \circ H_{9,1}$ with the number of vertices from 29112 to 54072.  Problems of this kind seem to be a way harder computationally than graph coloring problems mentioned above. Namely, when fixing some 5-coloring of subgraph $H_{12,1}$, the \textbf{Kissat} solver can't produce a result in ~1 CPU month. This can mean nothing but high complexity of the problem, which should not be surprising given such number of vertices. On the other hand, it is possible that using a slightly different geometrical construction will allow us to find a graph for which the absence of 5-coloring is easy to check.

Something similar can be said in the case of the plane. The problem of coloring 64513-vertex graph (section 6) with 5 colors is surprisingly difficult. But if we assume that the origin has color 5 and other vertices are colored with colors 1, 2, 3, 4, then the problem becomes easy. Apparently, the problem of finding a proper 5-coloring of such a graph constructed for $s>3$ may be practically unsolvable using pure SAT approach, but it hardly means anything when the number of vertices is so large.

It is noteworthy that the constructed 5-chromatic graphs can be realized only at a single value of surface curvature.  It is well known that a unit distance graph embedded in a sphere of continuously varying radius can have arbitrarily large mean degree  \cite{EHP89, Swanepoel}, but the known example  is a bipartite graph and it obviously cannot serve as a basis for construction of a 5-chromatic graph which can be embedded into a surface of continuously varying curvature. Therefore it appears to be true that the computational methods considered in this paper can at best be applied only to construct a finite family of 5-chromatic graphs at a finite set of curvature values.


\section*{Acknowledgments}

The authors would like to thank  A.~M. Raigorodskii for the source of motivation, O.~S. Zaikin and S.~E. Kochemazov for the tips on working with  SAT solvers, and the Regional center for supporting gifted children ``Polaris-Adygea'' for providing computing resources. The authors also thank the anonymous referees for helpful comments, Aubrey de Grey and Jaan Parts for pointing out some inaccuracies, and Marijn Heule for discussing the results for the case of the plane.

 \begin{table}[ht]
    \centering
    \begin{tabular}{|c|c|p{9cm}|}
    \hline
    & Value & Minimal polynomial for $2 x_i$, $2 y_i$,  $2 z_i$ \\
    \hline
    $x_{1}$ & 0.46246   & $-49 - 147 x - 71 x^2 + 67 x^3 + 129 x^4 + 97 x^5 + 42 x^6 + 10 x^7 + x^8$ \\
    $y_{1}$ & -0.32804   &$ -19 + 57 x + 206 x^2 + 117 x^3 - 11 x^4 - 23 x^5 + x^6 + 2 x^7 + x^8$ \\
    $z_{1}$ &-0.15496    & $25+150 x+315 x^2+370 x^3+284 x^4+ 145 x^5 + 49 x^6 + 10 x^7 + x^8$\\
    \hline
    $x_{2}$ &-0.22176    &$-49-147 x- 71 x^2 + 67 x^3 + 129 x^4 + 97 x^5 + 42 x^6 + 10 x^7 + x^8$ \\
    $y_{2}$ &0.09483    &$-19 + 57 x + 206 x^2 + 117 x^3 - 11 x^4 - 23 x^5 + x^6 + 2 x^7 + x^8$ \\
    $z_{2}$ &0.53602    & $25-150 x+315 x^2-370 x^3+284 x^4- 145 x^5 + 49 x^6 -10 x^7 + x^8$\\
    \hline
    $x_{3}$ &-0.57219    &$1 + x - 16 x^2 - 21 x^3 + 34 x^4 + 156 x^5 + 234 x^6 +159 x^7 + 41 x^8$ \\
    $y_{3}$ &-0.04462    &$-1 - 15 x - 55 x^2 - 121 x^3 + 269 x^4 + 690 x^5 + 616 x^6 + 257 x^7 + 41 x^8$ \\
    $z_{3}$ &-0.12687    &$-41 - 211 x - 231 x^2 - 148 x^3 - 13 x^4 + 75 x^5 + 223 x^6 + 174 x^7 + 41 x^8$ \\
    \hline
    $x_{4}$ &-0.13592    &$1 + x - 16 x^2 - 21 x^3 + 34 x^4 + 156 x^5 + 234 x^6 +159 x^7 + 41 x^8$ \\
    $y_{4}$ &-0.22502    &$-1+15 x-55 x^2+121 x^3+269 x^4-690 x^5+ 616 x^6 - 257 x^7 + 41 x^8$ \\
    $z_{4}$ &-0.52572    &$-41+211 x-231 x^2+148 x^3-13 x^4-75 x^5+223 x^6 - 174 x^7 + 41 x^8$ \\
    \hline
    $x_{5}$ &0.07219    &$-1 + 5 x + 10 x^2 + 8 x^3 + 69 x^4 + 205 x^5 + 269 x^6 + 169 x^7 + 41 x^8$  \\
    $y_{5}$ &0.35363    &$1 - 13 x + 61 x^2 - 138 x^3 + 164 x^4 - 118 x^5 + 96 x^6 - 93 x^7 + 41 x^8$ \\
    $z_{5}$ &0.46392    &$269 - 232 x - 669 x^2 + 471 x^3 + 642 x^4 - 335 x^5 - 258 x^6 + 72 x^7 + 41 x^8$
    \\
    \hline
    $x_{6}$ &-0.36408    &$-1 + 5 x + 10 x^2 + 8 x^3 + 69 x^4 + 205 x^5 + 269 x^6 + 169 x^7 + 41 x^8$ \\
    $y_{6}$ &0.08400    &$1 - 13 x + 61 x^2 - 138 x^3 + 164 x^4 - 118 x^5 + 96 x^6 - 93 x^7 + 41 x^8$ \\
    $z_{6}$ &-0.45374    &$269 + 232 x - 669 x^2 - 471 x^3 + 642 x^4 + 335 x^5 - 258 x^6 - 72 x^7 + 41 x^8$ \\
    \hline
    $x_{7}$ &-0.58345    &$-1 + 4 x - 2 x^2 - 12 x^3 - x^4 + 15 x^5 + 15 x^6 + 6 x^7 + x^8$ \\
    $y_{7}$ &0.05157    &$-1 + 10 x - 5 x^2 + 18 x^3 + 19 x^4 - 15 x^5 - x^6 - x^7 + x^8$ \\
    $z_{7}$ &0.0492    &$1 - 14 x + 36 x^2 + 29 x^3 + 13 x^4 + 17 x^5 + 19 x^6 + 7 x^7 + x^8$ \\
    \hline
    $x_{8}$ &-0.4376    &$-1 - x - 2 x^2 + 4 x^3 + 14 x^4 + 15 x^5 + 13 x^6 + 6 x^7 + x^8$ \\
    $y_{8}$ & 0.03856    &$1 - 15 x + 28 x^2 - 21 x^3 - x^4 + 11 x^5 + x^6 - 4 x^7 + x^8$ \\
    $z_{8}$ & 0.39052    & $-41 - 37 x + 108 x^2 + 45 x^3 - 42 x^4 - 11 x^5 + 4 x^6 + 2 x^7 + x^8 $\\
    \hline
    \end{tabular}
    \caption{Minimal polynomials for vertex coordinates  of $G_{972}$}
    \label{coord2}
\end{table}

\newpage

\bibliographystyle{abbrv}
\bibliography{main.bib}

\begin{thebibliography}{10}

\bibitem{Agoston}
P.~{\'A}goston.
\newblock A lower bound on the number of colours needed to nicely colour a
  sphere.
\newblock In {\em 32nd Canadian Conference on Computational Geometry}, pages
  273--284, 2020.

\bibitem{Alon}
N.~Alon and A.~Kupavskii.
\newblock Two notions of unit distance graphs.
\newblock {\em Journal of Combinatorial Theory, Series A}, 125:1--17, July
  2014.

\bibitem{Glucose}
G.~Audemard and L.~Simon.
\newblock On the glucose {SAT} solver.
\newblock {\em Int. J. Artif. Intell. Tools}, 27(1):1840001:1--1840001:25,
  2018.

\bibitem{Ballard}
D.~Ballard.
\newblock Coloring a dodecahedron with four colors.
\newblock {\em Canadian Mathematical Bulletin}, 14(1):103--105, Jan. 1971.

\bibitem{Biere1}
A.~Biere.
\newblock Lingeling, {Plingeling} and {Treengeling} entering the {SAT
  Competition 2013}.
\newblock {\em Proceedings of SAT competition}, 2013:1, 2013.

\bibitem{Biere2}
A.~Biere, K.~Fazekas, M.~Fleury, and M.~Heisinger.
\newblock {CaDiCaL}, {Kissat}, {Paracooba}, {Plingeling} and {Treengeling}
  entering the {SAT Competition 2020}.
\newblock In T.~Balyo, N.~Froleyks, M.~Heule, M.~Iser, M.~J{\"a}rvisalo, and
  M.~Suda, editors, {\em Proc.~of {SAT Competition} 2020 — Solver and
  Benchmark Descriptions}, volume B-2020-1 of {\em Department of Computer
  Science Report Series B}, pages 51--53. University of Helsinki, 2020.

\bibitem{Rai3}
D.~Cherkashin, A.~Kulikov, and A.~Raigorodskii.
\newblock On the chromatic numbers of small-dimensional euclidean spaces.
\newblock {\em Discrete Applied Mathematics}, 243:125--131, 2018.

\bibitem{Coulson}
D.~Coulson.
\newblock A 15-colouring of 3-space omitting distance one.
\newblock {\em Discrete Mathematics}, 256(1-2):83--90, Sept. 2002.

\bibitem{Erdos}
N.~G. de~Bruijn and P.~Erdos.
\newblock A colour problem for infinite graphs and a problem in the theory of
  relations.
\newblock {\em Indigationes Mathematicae}, 13:371--373, 1951.

\bibitem{deGrey}
A.~D. N.~J. de~Grey.
\newblock The chromatic number of the plane is at least 5.
\newblock {\em Geombinatorics}, 28/1:18--31, 2018.

\bibitem{Minisat}
N.~E{\'{e}}n and N.~S{\"{o}}rensson.
\newblock An extensible {SAT}-solver.
\newblock In E.~Giunchiglia and A.~Tacchella, editors, {\em Theory and
  Applications of Satisfiability Testing, 6th International Conference, {SAT}
  2003. Santa Margherita Ligure, Italy, May 5-8, 2003 Selected Revised Papers},
  volume 2919 of {\em Lecture Notes in Computer Science}, pages 502--518.
  Springer, 2003.

\bibitem{EHP89}
P.~Erd{\"o}s, D.~Hickerson, and J.~Pach.
\newblock A problem of {Leo} {Moser} about repeated distances on the sphere.
\newblock {\em The American Mathematical Monthly}, 96(7):569--575, 1989.

\bibitem{Exoo2}
G.~Exoo and D.~Ismailescu.
\newblock The chromatic number of the plane is at least 5: A new proof.
\newblock {\em Discrete and Computational Geometry}, 64(1):216--226, 2020.

\bibitem{Godsil}
C.~D. Godsil and J.~Zaks.
\newblock Colouring the sphere.
\newblock {\em arXiv preprint arXiv:1201.0486}, 2012.

\bibitem{combrig}
J.~E. Graver, B.~Servatius, and H.~Servatius.
\newblock {\em Combinatorial rigidity}.
\newblock Number~2 in Graduate Stud. Math. American Mathematical Soc., 1993.

\bibitem{Hadwiger}
H.~Hadwiger.
\newblock Ein berdeckungssätze für den {Euklidischen} {Raum}.
\newblock {\em Portugaliae mathematica}, 4(3):140--144, 1944.

\bibitem{Henneberg}
L.~Henneberg.
\newblock {\em Die graphische Statik der starren Systeme}, volume~31.
\newblock BG Teubner, 1911.

\bibitem{Heule1}
M.~J.~H. Heule.
\newblock Computing small unit-distance graphs with chromatic number 5.
\newblock {\em Geombinatorics}, 28/1:32--50, 2018.

\bibitem{igraphm}
S.~Horv{\'a}t.
\newblock {The IGraph/M library for the Wolfram Mathematica}.
\newblock \url{https://github.com/szhorvat/IGraphM}, 2020.

\bibitem{Zaikin}
S.~Kochemazov, O.~Zaikin, A.~A. Semenov, and V.~Kondratiev.
\newblock Speeding up {CDCL} inference with duplicate learnt clauses.
\newblock In G.~D. Giacomo, A.~Catal{\'{a}}, B.~Dilkina, M.~Milano, S.~Barro,
  A.~Bugar{\'{\i}}n, and J.~Lang, editors, {\em {ECAI} 2020 - 24th European
  Conference on Artificial Intelligence, 29 August-8 September 2020, Santiago
  de Compostela, Spain, August 29 - September 8, 2020 - Including 10th
  Conference on Prestigious Applications of Artificial Intelligence {(PAIS}
  2020)}, volume 325 of {\em Frontiers in Artificial Intelligence and
  Applications}, pages 339--346. {IOS} Press, 2020.

\bibitem{Kostina}
O.~Kostina.
\newblock On lower bounds for the chromatic number of spheres.
\newblock {\em Mathematical Notes}, 105(1-2):16--27, Jan. 2019.

\bibitem{Laman}
G.~Laman.
\newblock On graphs and rigidity of plane skeletal structures.
\newblock {\em Journal of Engineering Mathematics}, 4(4):331--340, Oct. 1970.

\bibitem{Larman}
D.~Larman and C.~Rogers.
\newblock The realization of distances within sets in euclidean space.
\newblock {\em Mathematika}, 19(1):1--24, June 1972.

\bibitem{painless}
L.~{Le Frioux}, S.~Baarir, J.~Sopena, and F.~Kordon.
\newblock Modular and efficient divide-and-conquer {SAT} solver on top of the
  {Painless} framework.
\newblock In {\em Proceedings of the 25th International Conference on Tools and
  Algorithms for the Construction and Analysis of Systems (TACAS'19)}, volume
  11427 of {\em Lecture Notes in Computer Science}, pages 135--151. Springer,
  Cham, Apr. 2019.

\bibitem{Lovasz}
L.~Lov{\'a}sz.
\newblock Self-dual polytopes and the chromatic number of distance graphs on
  the sphere.
\newblock {\em Acta Sci. Math.(Szeged)}, 45(1-4):317--323, 1983.

\bibitem{CDCL}
J.~P. Marques-Silva, I.~Lynce, and S.~Malik.
\newblock Conflict-driven clause learning {SAT} solvers.
\newblock In A.~Biere, M.~Heule, H.~van Maaren, and T.~Walsh, editors, {\em
  Handbook of Satisfiability}, volume 185 of {\em Frontiers in Artificial
  Intelligence and Applications}, pages 131--153. {IOS} Press, 2009.

\bibitem{Nechushtan}
O.~Nechushtan.
\newblock On the space chromatic number.
\newblock {\em Discrete Mathematics}, 256(1-2):499--507, Sept. 2002.

\bibitem{Parts2}
J.~Parts.
\newblock The chromatic number of the plane is at least 5 --- a
  human-verifiable proof.
\newblock {\em Geombinatorics}, 30/2:77--102, 2020.

\bibitem{Parts1}
J.~Parts.
\newblock Graph minimization, focusing on the example of 5-chromatic
  unit-distance graphs in the plane.
\newblock {\em Geombinatorics}, 29/4:137--166, 2020.

\bibitem{Prosanov}
R.~Prosanov.
\newblock Chromatic numbers of spheres.
\newblock {\em Discrete Mathematics}, 341(11):3123--3133, Nov. 2018.

\bibitem{Rai1}
A.~M. Raigorodskii.
\newblock On the chromatic number of a space.
\newblock {\em Russian Mathematical Surveys}, 55(2):351--352, 2000.

\bibitem{Rai2}
A.~M. Raigorodskii.
\newblock On the chromatic numbers of spheres in $\mathbb{R}^n$.
\newblock {\em Combinatorica}, 32(1):111--123, 2012.

\bibitem{oeis2}
\relax{OEIS Foundation Inc.}
\newblock {Number of 4-chromatic Laman graphs on n vertices}.
\newblock \url{http://oeis.org/A328061}, 2019.

\bibitem{Simmons}
G.~J. Simmons.
\newblock The chromatic number of the sphere.
\newblock {\em Journal of the Australian Mathematical Society}, 21(4):473--480,
  June 1976.

\bibitem{Sirgedas}
T.~Sirgedas.
\newblock The surface of a sufficiently large sphere has chromatic number at
  most 7.
\newblock {\em Geombinatorics}, 30:138--151, 2021.

\bibitem{Soifer}
A.~Soifer.
\newblock {\em The mathematical coloring book: Mathematics of coloring and the
  colorful life of its creators}.
\newblock Springer Science \& Business Media, 2008.

\bibitem{Swanepoel}
K.~J. Swanepoel and P.~Valtr.
\newblock The unit distance problem on spheres.
\newblock {\em Contemporary Mathematics}, 342:273--280, 2004.

\bibitem{Github}
V.~A. Voronov.
\newblock {The github repository}.
\newblock \url{https://github.com/vsvor/dist-graphs/}, 2020.

\end{thebibliography}

\end{document}